\documentclass[letterpaper,12pt]{amsart}

\title[L-Infinity optimization to Bergman fans]{L-Infinity optimization to Bergman fans of matroids with an application to phylogenetics}

\author{Daniel Irving Bernstein}
\address{Institute for Data, Systems, and Society, Massachusetts Institute of Technology, 77 Massachusetts Avenue, Cambridge, MA 02139}
\email{\url{dibernst@mit.edu}, \url{bernstein.daniel@gmail.com}}

\usepackage{latexsym,array,delarray,amsthm,amssymb,epsfig,graphicx,multirow}
\usepackage{algorithmic}
\usepackage[section]{algorithm}
\usepackage{amsmath}
\usepackage[numbers]{natbib}
\usepackage{tikz}
\usepackage{verbatim}
\usepackage{graphicx}
\usepackage{subcaption}
\usepackage{mathrsfs}

\theoremstyle{plain}
\newtheorem{thm}{Theorem}[section]
\newtheorem{lemma}[thm]{Lemma}
\newtheorem{prop}[thm]{Proposition}

\newtheorem*{thm*}{Theorem}
\newtheorem*{lemma*}{Lemma}
\newtheorem*{prop*}{Proposition}
\newtheorem*{cor*}{Corollary}
\newtheorem*{conj*}{Conjecture}

\theoremstyle{definition}
\newtheorem{defn}[thm]{Definition}
\newtheorem*{defn*}{Definition}
\newtheorem{ex}[thm]{Example}

\newtheorem{rmk}[thm]{Remark}

\usepackage{thmtools, thm-restate}
\usepackage{longtable}
\usepackage{hyperref}
\usepackage{cleveref}

\newcommand{\rr}{\mathbb{R}}

\newcommand{\calm}{\mathcal{M}}
\newcommand{\caln}{\mathcal{N}}

\newcommand{\calt}{\mathcal{T}}
\newcommand{\cals}{\mathcal{S}}

\newcommand{\ind}{\mbox{$\perp \kern-5.5pt \perp$}}

\newcommand{\rank}{\textnormal{rank}}

\newcommand{\rmax}{\ensuremath \rr_\textnormal{max}}
\newcommand{\tconv}{\textnormal{tconv}}
\newcommand{\tcone}{\textnormal{tcone}}

\newcommand{\berg}{\tilde{\mathcal{B}}}
\newcommand{\nest}{\tilde N}

\renewcommand{\root}{\textnormal{root}}

\newcommand{\des}{\textnormal{Des}}

\tikzstyle{vertex}=[circle, draw, inner sep=0pt, minimum size=6pt, fill=black]
\newcommand{\vertex}{\node[vertex]}

\begin{document}

\begin{abstract}
    Given a dissimilarity map $\delta$ on finite set $X$,
    the set of ultrametrics (equidistant tree metrics)
    which are $l^\infty$-nearest to $\delta$ is a tropical polytope.
    We give an internal description of this tropical polytope
    { which we use to derive a polynomial-time checkable test
    for the condition that all ultrametrics $l^\infty$-nearest to $\delta$ have the same tree structure}.
    It was shown by Ardila and Klivans \cite{ardila-klivans2006}
    that the set of all ultrametrics on a finite set of size $n$
    is the Bergman fan associated to the matroid underlying the complete graph
    on $n$ vertices.
    Therefore, we derive our results in the more general context of Bergman fans
    of matroids.
    This added generality allows our results to be used on
    dissimilarity maps where only a subset of the entries are known.


    \smallskip
    \noindent \textbf{MSC Classes:} 14T05, 05B35
\end{abstract}

\maketitle

\section{Introduction}
    A fundamental problem in phylogenetics is to infer the evolutionary history
    among a set of genes or species from data.
    One approach is to use \emph{distance-based methods}.
    The data required for such an approach is some measure of distance between each pair of species.
    If these distances are computed using some property that is expected to change in proportion to time elapsed,
    then one often assumes that the pairwise distances approximate an ultrametric.
    Finding a best-fit ultrametric to an arbitrary dissimilarity map is therefore
    an important computational problem.
    For background, see \cite[Chapter~7]{Semple2003}.

    {
    A major source of difficulty in this endeavor stems from the fact that
    two of the most basic sets with which one would reason about distance-based phylogenetics,
    namely the set of tree metrics and the set of ultrametrics,
    do not interact with Euclidean geometry in a clean way,
    thus making naive application of traditional statistical methods problematic.
    Beginning with work of Billera, Holmes, and Vogtman \cite{billera2001geometry},
    the past two decades have seen much research into developing and studying geometric theories
    that interact nicely with the sets of tree metrics and ultrametrics,
    with the hope that reinterpreting traditional statistical theory and methods in these new geometries
    will lead to something useful.
    Speyer and Sturmfels \cite{Speyer} and Ardila and Klivans \cite{ardila-klivans2006}
    showed that the sets of tree metrics and ultrametrics are tropical varieties,
    thus giving the first indication that tropical geometry might offer useful tools for phylogenetics.
    }

    {
    Since then, researchers have been exploring tropical geometry's potential 
    as a fundamental theory on which to develop statistical methods designed specifically for phylogenetic applications.
    Tropical geometry is a geometric theory that one naturally obtains
    when redefining arithmetic over $\rr \cup \{-\infty\}$ so that the sum of two
    numbers is their maximum and the product is their sum (in the usual sense).
    The natural choice for a metric in tropical geometry is the $l^\infty$-metric.
    A recent preprint of Lin, Monod, and Yoshida \cite{lin2018tropical} shows that
    the set of phylogenetic trees endowed with the tropically projectivized $l^\infty$-metric,
    which they call \emph{palm tree space},
    has many features of Euclidean space that enable classical statistical theory to work.
    In particular, palm tree space supports probability measures and a reasonable theory of linear algebra.
    }In \cite{lin-sturmfels2016},
    Lin, Sturmfels, Tang, and Yoshida compare tropical convexity
    to the convexity theory of Billera, Holmes, and Vogtman \cite{billera2001geometry}
    with regard to their potential as theoretical frameworks
    for developing algorithms to reduce
    the complexity of a dataset consisting of several ultrametrics on the same taxa.
    { They show that in the convexity theory of \cite{billera2001geometry},
    a triangle (i.e. the convex hull of three points) can have arbitrarily high dimension,
    whereas triangles are always two-dimensional in the projective tropical setting \cite{develin-sturmfels2004}.}
    In \cite{yoshida-zhang2017tropical},
    Yoshida, Zhang, and Zhang develop a theory of tropical principal component analysis.
    
    { A recurring frustration one encounters when tropicalizing a classical object
    is that uniqueness guarantees may disappear.
    In order for tropical geometry to be considered a reasonable mathematical foundation for phylogenetic analysis,
    failures of uniqueness must be understood when they have potential to cause problems.
    Lin and Yoshida \cite{lin2018tropicalFermat} studied non-uniqueness of the tropical Fermat-Weber point,
    which is analogous to the geometric mean from Euclidean geometry.
    They showed that the set of all tropical Fermat-Weber points is a (classical) polytope,
    and gave a necessary condition for uniqueness of the tropical Fermat-Weber point.
    In this paper, we provide analogous results for
    non-uniqueness of the ultrametric that is nearest to a given dissimilarity map in the $l^\infty$-metric.
    Colby Long and this author began a study of this, and other related phenomena, in \cite{bernstein2017infinity}.}
    The main mathematical results of \cite{bernstein2017infinity} concern the non-uniqueness of the point in a (non-tropical)
    linear subspace of $\rr^n$ that is $l^\infty$-nearest to a given $x \in \rr^n$.
    In that paper, it is also shown that there exist dissimilarity maps in $\rr^{\binom{n}{2}}$
    whose set of $l^\infty$-nearest ultrametrics contains $\frac{1}{3}\cdot(2n-3)!!$ different
    tree topologies.

    This paper builds on some of these observations.
    In particular, Proposition \ref{prop:ultrametricTropicalPolytope} says that the set of
    ultrametrics $l^\infty$ nearest to a given dissimilarity map is a tropical polytope,
    Theorem \ref{thm:treesAlgorithm} provides an internal description,
    and Theorem \ref{thm:sameTopologyTest} gives a polynomial-time checkable condition,
    telling us exactly when all nearest ultrametrics have the same tree structure.
    From an phylogenetics perspective,
    this is useful information since the tree structure describes the evolutionary
    relationship among the species being studied.

    We derive our results in a more general context.
    Ardila and Klivans showed that the set of ultrametrics on $n$ species is the Bergman fan
    associated to the matroid underlying the complete graph on $n$ vertices \cite{ardila-klivans2006}.
    Therefore we can view the problem of finding the set of $l^\infty$-nearest ultrametrics
    as a special case of the problem of finding
    the set of $l^\infty$-nearest points in the Bergman fan of a matroid.
    This latter set is also a tropical polytope (Proposition \ref{prop:tropicalPolytope})
    and Theorem \ref{thm:verticesAlgorithm} provides an internal description.
    Feichtner and Sturmfels describe a refinement of the Bergman fan underlying a matroid
    \cite{feichtner-sturmfels2005} which can be used to generalize the concept of tree topology.
    In light of this, Theorem \ref{thm:verticesAlgorithm} is the
    straightforward generalization of Theorem \ref{thm:treesAlgorithm}.

    {The added generality of Bergman fans of matroids has a potential application in phylogenetics.}
    Namely, if one wishes to reconstruct a phylogeny from partial distance data
    where observed distances correspond to the edges of some graph $G$,
    then one can begin by optimizing to the Bergman fan of $G$'s matroid
    which will give a partial ultrametric (see Proposition \ref{prop:completion}).
    This reconstruction problem is a special case of the \emph{sandwich to ultrametric problem}
    studied by Farach, Kannan, and Warnow in \cite{farach1993robust}.
    { The added generality is also interesting from a pure tropical geometry perspective.
    In particular, given the Bergman fan $\berg(\calm)$ of a matroid $\calm$,
    the question of describing points in $\berg(\calm)$ that are tropically nearest to a given $x \notin \berg(\calm)$
    is in some sense the tropical analog of finding the point of a (classical) linear space $L$ that is Euclidean-nearest
    to a given $x \notin L$.
    }


    Just as with ordinary polytopes, tropical polytopes admit \emph{external descriptions}
    as the intersection of tropical half-spaces,
    as well as internal descriptions \cite{gaubert2011minimal}.
	Theorem~7.1 in \cite{akian2011best} can be used to obtain an external description
	of the tropical polytopes we are interested in.
	However, an internal description is more advantageous for our purposes
	because it gives us a way to check whether all ultrametics
	$l^\infty$-nearest to a given dissimilarity map have the same tree topology
	(see Theorem \ref{thm:sameTopologyTest} and Proposition \ref{prop:multipletopologies}).  

    This paper is organized as follows.
    Section \ref{sec:tropicalConvexity} gives the necessary background on tropical convexity.
    Section \ref{sec:phylogenetics} contains Theorem~\ref{thm:treesAlgorithm},
    which is an internal description
    of the tropical polytope consisting of the ultrametrics that are $l^\infty$-nearest
    to a given dissimilarity map.
    {
    A proof is deferred until Section \ref{sec:bergmanOptimization}.
    Section~\ref{sec:phylogenetics} also states and proves Theorem \ref{thm:sameTopologyTest},
    which provides a polynomial-time method for checking that all ultrametrics $l^\infty$-nearest
    to a given dissimilarity map have the same tree topology}.
    Section \ref{sec:nestedSets} uses results of Feichtner and Sturmfels \cite{feichtner-sturmfels2005} to 
    generalize the tree structure underlying
    an ultrametric to a similar combinatorial structure underlying an element
    of the Bergman fan of an arbitrary matroid.
    This combinatorial structure is used in Section \ref{sec:bergmanOptimization}
    to generalize Theorem~\ref{thm:treesAlgorithm} to get Theorem~\ref{thm:verticesAlgorithm}.
    Section \ref{sec:data} applies Theorem~\ref{thm:treesAlgorithm} to a biological dataset.
\section*{Acknowledgments}
    The author is grateful to Colby Long and Seth Sullivant
    for many helpful conversations and for feedback on early drafts,
    and to several anonymous referees who provided thoughtful feedback that greatly improved this manuscript.
    This work was partially supported by the US National Science Foundation (DMS 0954865 and 1802902) and the David and Lucille Packard Foundation.


\section{Preliminaries on Tropical Convexity}\label{sec:tropicalConvexity}
This section reviews the necessary concepts from tropical convexity.
There are at least two different sets of basic definitions related to tropical convexity.
One is used in \cite{develin-sturmfels2004}, and the other in \cite{allamigeon-gaubert-goubault2010}.
We adhere to the conventions of the latter, as their definition of tropical polytope
is more natural in our context.
\\
\indent
The \emph{tropical semiring}, also known as the max-plus algebra, is the set $\rr \cup \{-\infty\}$ together with the operations
$a \oplus b := \max\{a,b\}$ and $a \odot b := a + b$.
We denote this semiring by $\rmax$.
The additive identity of $\rmax$ is $-\infty$
and the multiplicative identity is $0$.
The set $\rmax^n$ is an $\rmax$-semimodule where for $x,y \in \rmax^n$ and $\alpha \in \rmax$,
$(x \oplus y)_i := x_i \oplus y_i$ and $(\alpha \odot x)_i := \alpha + x_i$.
If $A \in \rmax^{m \times n}$ is a matrix and $x \in \rmax^n$,
then the product $A \odot x$ is the usual matrix product,
but with multiplication and addition interpreted tropically.
That is, if $A$ has columns $a_1,\dots,a_n$, then
\[
	A \odot x:= \bigoplus_{j=1}^n x_j \odot a_j.
\]

Several notions from ordinary convexity theory have tropical analogs.
We say that $P \subseteq \rmax$ is a \emph{tropical cone} if whenever $x,y \in P$ and $\lambda,\mu \in \rmax$,
$\lambda \odot x \oplus \mu \odot y \in P$.
If this only holds with the restriction that $\lambda \oplus \mu = 0$,
then we say that $P$ is \emph{tropically convex}.
A \emph{tropical polyhedron} is a set of the form
\[
	\{x \in \rmax^n : A \odot x \oplus b \ge C \odot x \oplus d\}
\]
where $A,C \in \rmax^{m \times n}$ and $b,d \in \rmax^m$.
We denote this set $P(A,b,C,d)$.
It follows from the discussion below that $P(A,b,C,d)$ is tropically convex.
When $b = d = (-\infty,\dots,-\infty)^T$ then $P(A,b,C,d)$ is a tropical cone
and we call it a \emph{tropical polyhedral cone}.
Bounded tropical polyhedra are called \emph{tropical polytopes}.
Given $V \subseteq \rmax$,
$\tconv(V)$ is the \emph{tropical convex hull} of $V$.
That is,
\[
	\tconv(V) := \{\lambda \odot x + \mu \odot y : x,y \in V, \lambda \oplus \mu = 0\}.
\]
We define the \emph{tropical conic hull} $\tcone(V)$ similarly.
Gaubert and Katz showed in \cite{gaubert-katz2007} that any tropical polytope (cone) $P$
can be expressed as the tropical convex (conic) hull of a finite set $V$.
Conversely, Gaubert showed that if $V \subseteq \rr^n$ is a finite set and $P = \tcone(V)$,
then $P$ is a tropical polyhedral cone \cite[Corollary 1.2.5]{gaubert1992theorie}
The analogous result for $P = \tconv(V)$ follows from results in \cite{gaubert-katz2007}.
{ There exists a minimal such $V$ (see \cite[Theorem 18]{butkovivc2007generators} or \cite[Theorem 3.1]{gaubert-katz2007}) called the \emph{tropical vertices (extreme rays) of $P$}.
See also \cite[Theorem~2]{gaubert2011minimal}.}


\section{Results for phylogenetics: l-infinity nearest ultrametrics}\label{sec:phylogenetics}
This section presents the results of Section \ref{sec:bergmanOptimization}
in the context of our main motivation.
In particular, Theorem \ref{thm:treesAlgorithm} gives a combinatorial description
of a finite set of ultrametrics whose tropical convex hull
is the set of ultrametrics nearest in the $l^\infty$-norm to a given dissimilarity map.
{ We also use Theorem \ref{thm:treesAlgorithm} to derive Theorem \ref{thm:sameTopologyTest},
which gives a polynomial-time checkable condition guaranteeing that all ultrametrics $l^\infty$
nearest to a given dissimilarity map have the same tree topology.}
We begin by reviewing the necessary background from \cite{Semple2003} about ultrametrics,
which are a special type of tree metric.
\\
\indent
Let $X = \{x_1,\dots,x_n\}$ be a finite set.
A \emph{dissimilarity map} on $X$ is a function $\delta: X \times X \rightarrow \rr$
such that $\delta(x,x) = 0$ and $\delta(x,y) = \delta(y,x)$ for all $x,y \in X$.
Note that we allow dissimilarity maps to take negative values.
We can express a dissimilarity map $d$
as a matrix $D$ where $D_{ij} = \delta(x_i,x_j)$.
Note that $D$ is symmetric with zeros along the diagonal.
A \emph{rooted $X$-tree} is a tree with leaf set $X$ 
where one interior (i.e. non-leaf) vertex has been designated
the ``root.''
We use the notation $\root(T)$ for the root of a rooted $X$-tree $T$.
A \emph{descendant} of a vertex $v$ in a rooted tree $T$
is a node $u \neq v$ such that the unique path from $u$ to $\root(T)$ contains $v$.
Note that all non-root vertices are descendants of $\root(T)$.
The set of descendants of a vertex $v$ in a rooted tree $T$ is denoted $\des_T(v)$.
We let $T^\circ$ denote the set of interior vertices of $T$.

Let $T$ be a rooted $X$-tree and let
$\alpha:T^\circ \rightarrow \rr$ be a weighting of the internal nodes of $T$.
We say that $\alpha$ is \emph{compatible with $T$} if $\alpha(u) \le \alpha(v)$ whenever $u \in \des_T(v)$.
The pair $(T,\alpha)$ gives rise to a dissimilarity map $\delta_{T,\alpha}$ on $X$
defined by $\delta_{T,\alpha}(x_i,x_j) := \alpha(v)$ where $v \in T^\circ$
is the vertex nearest to $\root(T)$ in the unique path from $x_i$ to $x_j$.
Given a dissimilarity map $\delta$ on $X$,
if we can express $\delta$ as $\delta_{T,\alpha}$ for some $X$-tree $T$
and compatible internal node weighting $\alpha$,
then $\delta$ is said to be an \emph{ultrametric}.
If we require that $\alpha(u) < \alpha(v)$ whenever $u \in \des_T(v)$,
then the rooted $X$-tree $T$ is unique
and called the \emph{(tree) topology} of $\delta$.
{ Some readers from other areas of mathematics take issue with this use of the word ``topology,''
but it is standard in the phylogenetics literature \cite{Semple2003}.}
Figure \ref{fig:ultrametricExample} shows an ultrametric
along with an interior-vertex-weighted tree displaying it.
\\
\indent
Some readers may be familiar with a seemingly different definition of ultrametric
which says that $\delta: X \times X \rightarrow \rr$ is an ultrametric if and only if
for every triple $x,y,z \in X$ of distinct elements,
the maximum of $\delta(x,y),\delta(x,z),\delta(y,z)$ is attained twice.
This is equivalent to the definition given above.
Sometimes the requirement that $x,y,z$ be distinct is relaxed.
This gives the more restricted class of ultrametrics,
consisting only of ultrametrics representable as $\delta_{T,\alpha}$ for nonnegative $\alpha$
compatible with $T$.
See \cite[Chapter~7]{Semple2003} for details.
We use the more inclusive definition of an ultrametric because it
simplifies connections with tropical geometry.
\\
\indent
A \emph{polytomy} of a rooted tree is either a non-root internal node of degree at least four,
or the root node if it has degree at least three.
We say that a rooted tree is \emph{binary} if it does not have any polytomy.
A \emph{resolution} of a tree $T$ is a binary tree $T'$
such that $T$ can be obtained from $T'$ via a (possibly empty) series of edge contractions \cite{Semple2003}.
Note that if the topology underlying $\delta$ is not binary,
then there will be multiple resolutions of the topology of $\delta$.
Figure \ref{fig:polytomyTree} illustrates these concepts
by representing a single ultrametric
in three ways - on its topology
and on two different resolutions.

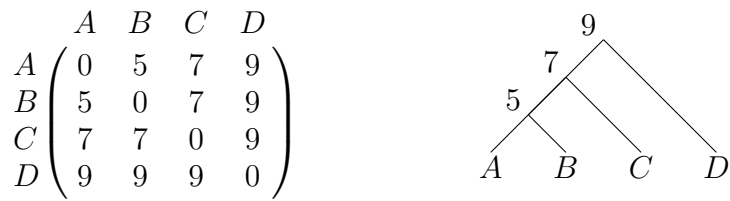
\begin{figure}
    \begin{subfigure}{0.49\textwidth}\centering
    	\bordermatrix{
    		~&A&B&C&D\cr
    		A&0&5&7&9\cr
    		B&5&0&7&9\cr
    		C&7&7&0&9\cr
    		D&9&9&9&0\cr
    	}
    \end{subfigure}
    \begin{subfigure}{0.49\textwidth}\centering
    	\begin{tikzpicture}
            \node (a) at (0,-0.2){$A$};
            \node (b) at (1,-0.2){$B$};
            \node (c) at (2,-0.2){$C$};
            \node (d) at (3,-0.2){$D$};
            \node (ab) at (1/2-0.2,1/2+0.2){$5$};
            \node (abc) at (1-0.2,1+0.2){$7$};
            \node (abcd) at (3/2-0.2,3/2+0.2){$9$};
            \draw (0,0) -- (1/2,1/2);
            \draw (1,0) -- (1/2,1/2);
            \draw (1/2,1/2) -- (1,1);
            \draw (2,0) -- (1,1);
            \draw (3,0) -- (3/2,3/2);
            \draw (1/2,1/2) -- (3/2,3/2);
    	\end{tikzpicture}
    \end{subfigure}
    \caption{An ultrametric on $\{A,B,C,D\}$ and its representation on a rooted tree.}\label{fig:ultrametricExample}
\end{figure}

\begin{figure}
    \begin{subfigure}{0.32\textwidth}\centering
        \begin{tikzpicture}
            \node (a) at (0,-0.2){$A$};
            \node (b) at (1,-0.2){$B$};
            \node (c) at (2,-0.2){$C$};
            \node (d) at (3,-0.2){$D$};
            \node (e) at (4,-0.2){$E$};
            \node (abc) at (1-0.2,1+0.2){$1$};
            \node (abcde) at (2-0.2,2+0.2){$2$};
            \draw (0,0) -- (1,1);
            \draw (1,0) -- (1,1);
            \draw (2,0) -- (1,1);
            \draw (1,1) -- (2,2);
            \draw (3,0) -- (2,2);
            \draw (4,0) -- (2,2);
        \end{tikzpicture}
    \end{subfigure}
    \begin{subfigure}{0.32\textwidth}\centering
        \begin{tikzpicture}
            \node (a) at (0,-0.2){$A$};
            \node (b) at (1,-0.2){$B$};
            \node (c) at (2,-0.2){$C$};
            \node (d) at (3,-0.2){$D$};
            \node (e) at (4,-0.2){$E$};
            \node (ab) at (1/2-0.2,1/2+0.2){$1$};
            \node (de) at (7/2+0.2,1/2+0.2){$2$};
            \node (abc) at (1-0.2,1+0.2){$1$};
            \node (abcde) at (2-0.2,2+0.2){$2$};
            \draw (0,0) -- (1/2,1/2);
            \draw (1,0) -- (1/2,1/2);
            \draw (1/2,1/2) -- (1,1);
            \draw (2,0) -- (1,1);
            \draw (1,1) -- (2,2);
            \draw (3,0) -- (7/2,1/2);
            \draw (4,0) -- (7/2,1/2);
            \draw (7/2,1/2) -- (2,2);
        \end{tikzpicture}
    \end{subfigure}
    \begin{subfigure}{0.32\textwidth}\centering
        \begin{tikzpicture}
            \node (a) at (0,-0.2){$A$};
            \node (b) at (2,-0.2){$B$};
            \node (c) at (1,-0.2){$C$};
            \node (d) at (3,-0.2){$D$};
            \node (e) at (4,-0.2){$E$};
            \node (ab) at (1/2-0.2,1/2+0.2){$1$};
            \node (abc) at (1-0.2,1+0.2){$1$};
            \node (abcd) at (3/2-0.2,3/2+0.2){$2$};
            \node (abcde) at (2-0.2,2+0.2){$2$};
            \draw (0,0) -- (1/2,1/2);
            \draw (1,0) -- (1/2,1/2);
            \draw (1/2,1/2) -- (1,1);
            \draw (2,0) -- (1,1);
            \draw (1,1) -- (2,2);
            \draw (3,0) -- (3/2,3/2);
            \draw (4,0) -- (7/2,1/2);
            \draw (7/2,1/2) -- (2,2);
        \end{tikzpicture}
    \end{subfigure}
    \caption{An ultrametric whose topology has two polytomies.
    Above, we see it represented on its topology and on two different resolutions.}\label{fig:polytomyTree}
\end{figure}
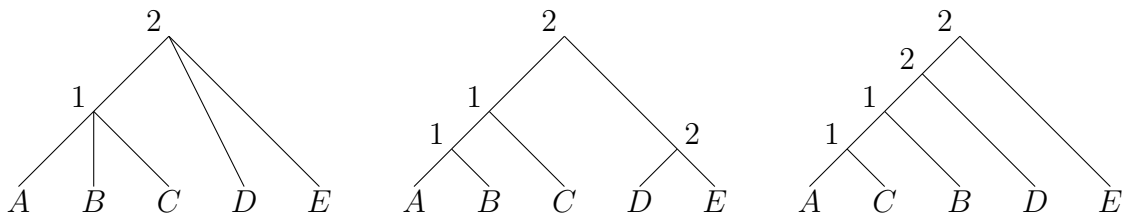

Given two dissimilarity maps $\delta_1,\delta_2$ on $X$
with associated matrices $D_1,D_2$,
we define the \emph{$l^\infty$ distance between $\delta_1$ and $\delta_2$},
denoted $\|\delta_1-\delta_2\|_\infty$,
to be the greatest absolute value among entries in $D_1-D_2$.
An important question that comes up in phylogenetics is then:
given a dissimilarity map $\delta$,
which ultrametrics are nearest to $\delta$ in the $l^\infty$ metric?
Chepoi and Fichet \cite{chepoi} give an algorithm for producing a single
ultrametric $l^\infty$-nearest to a given dissimilarity map
which we now describe.
We denote the all-ones vector or dissimilarity map by ${\bf 1}$.

\begin{thm}[{\cite[Corollary 1. See also discussion on p. 607]{chepoi}}]
\label{thm:chepoi}
    Let $\delta$ be a dissimilarity map on a finite set $X$.
    Then the following algorithm produces an ultrametric on $X$ that
    is nearest to $\delta$ in the $l^\infty$ norm.
    \begin{enumerate}
        \item Draw the complete graph on vertex set $X$.
        \item Label the edge between $x$ and $y$ by $\delta(x,y)$.
        \item Define $\delta_u: X \times X \rightarrow \rr$ so that for each $x,y \in X \times X$,
        \[
        	\delta_u(x,y) := \min_P
          	\left(\max_{\text{edges }(i,j) \text{ of } P} \delta(i,j)\right)
        \]
        where the minimum is taken over all paths $P$ from $x$ to $y$.
        \item Define $d := \|\delta_u-\delta\|_\infty$.
        Then $\delta_u + \frac{d}{2}{\bf 1}$ is an ultrametric that is $l^\infty$-nearest to $\delta$.
    \end{enumerate}
\end{thm}

Although the algorithm given by Theorem \ref{thm:chepoi}
produces only one ultrametric,
there can be multiple ultrametrics that are $l^\infty$-nearest
to a given dissimilarity map.
Figure \ref{fig:multipleClosest} shows a dissimilarity map alongside two
$l^\infty$-nearest ultrametrics with differing topologies.

\begin{defn}
    We call the ultrametric given by Theorem \ref{thm:chepoi}
    the \emph{maximal closest ultrametric to $\delta$} and denote it symbolically as $\delta_m$.
\end{defn}

That $\delta_m$ issue {coordinatewise-}maximal among ultrametrics nearest to $\delta$ is shown in \cite{chepoi},
and also follows from Lemma \ref{lem:maximalClosestMUltrametric}(\ref{item:maximal}).

\begin{figure}
    \begin{subfigure}{0.32\textwidth}\centering
        \bordermatrix{
            ~&A&B&C&D\cr
            A&0&2&4&6\cr
            B&2&0&8&10\cr
            C&4&8&0&12\cr
            D&6&10&12&0\cr
        }
    \end{subfigure}
    \begin{subfigure}{0.32\textwidth}\centering
        \begin{tikzpicture}
            \node (a) at (0,-0.2){$A$};
            \node (b) at (1,-0.2){$B$};
            \node (c) at (2,-0.2){$C$};
            \node (d) at (3,-0.2){$D$};
            \node (ab) at (1/2-0.2,1/2+0.2){$5$};
            \node (abc) at (1-0.2,1+0.2){$7$};
            \node (abcd) at (3/2-0.2,3/2+0.2){$9$};
            \draw (0,0) -- (1/2,1/2);
            \draw (1,0) -- (1/2,1/2);
            \draw (1/2,1/2) -- (1,1);
            \draw (2,0) -- (1,1);
            \draw (3,0) -- (3/2,3/2);
            \draw (1/2,1/2) -- (3/2,3/2);
        \end{tikzpicture}
    \end{subfigure}
    \begin{subfigure}{0.32\textwidth}\centering
        \begin{tikzpicture}
            \node (a) at (0,-0.2){$A$};
            \node (b) at (1,-0.2){$C$};
            \node (c) at (2,-0.2){$B$};
            \node (d) at (3,-0.2){$D$};
            \node (ab) at (1/2-0.2,1/2+0.2){$4$};
            \node (abc) at (1-0.2,1+0.2){$5$};
            \node (abcd) at (3/2-0.2,3/2+0.2){$9$};
            \draw (0,0) -- (1/2,1/2);
            \draw (1,0) -- (1/2,1/2);
            \draw (1/2,1/2) -- (1,1);
            \draw (2,0) -- (1,1);
            \draw (3,0) -- (3/2,3/2);
            \draw (1/2,1/2) -- (3/2,3/2);
        \end{tikzpicture}
    \end{subfigure}
    \caption{A dissimilarity map and two $l^\infty$-nearest ultrametrics with different topologies.}\label{fig:multipleClosest}
\end{figure}
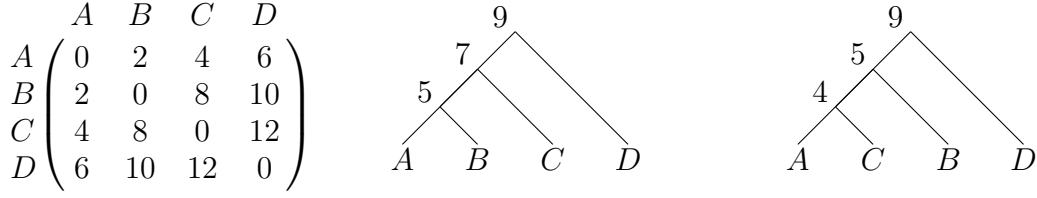

\begin{prop}\label{prop:ultrametricTropicalPolytope}
    Let $\delta$ be a dissimilarity map on a finite set $X$.
    The set of ultrametrics that are nearest to $\delta$ in the $l^\infty$-norm is a tropical polytope.
\end{prop}

We will prove Proposition \ref{prop:ultrametricTropicalPolytope} in a more general setting later
(see Proposition \ref{prop:tropicalPolytope}).
Given a dissimilarity map $\delta: X \times X \rightarrow \rr$,
Theorem \ref{thm:treesAlgorithm} describes a finite set of ultrametrics
whose tropical convex hull is the set of ultrametrics $l^\infty$-nearest to $\delta$.
The statement of Theorem \ref{thm:treesAlgorithm} requires the following definition.


\begin{defn}\label{defn:sliding}
    Let $\delta: X \times X \rightarrow \rr$
    be a dissimilarity map and let $u$ be an ultrametric
    that is closest to $\delta$ in the $l^\infty$-norm.
    Let $T$ be a resolution of the topology of $u$
    and let $\alpha: T^\circ \rightarrow \rr$
    be a compatible weighting of $T$'s internal nodes such that $\delta_{T,\alpha} = u$.
    An internal node $v$ of $T$ is said to be \emph{mobile} if there exists an
    ultrametric $\hat{u} \neq u$, expressible as $\hat u = \delta_{T,\hat{\alpha}}$
    for $\hat \alpha: T^\circ \rightarrow \rr$ such that
    \begin{enumerate}
        \item $\hat{u}$ is also nearest to $\delta$ in the $l^\infty$-norm,
        \item $\hat{\alpha}(x) = \alpha(x)$ for all internal nodes $x \neq v$, and
        \item $\hat{\alpha}(v) < \alpha(v)$.
    \end{enumerate}
    In this case, we say that $\hat u$ is obtained from $u$ by \emph{sliding $v$ down}.
    If moreover $v$ is no longer mobile in $\delta_{T,\hat{\alpha}}$, i.e. if
    $\hat{\alpha}(v) = \max\{\alpha(y): y \in \des_T(v)\}$, 
    or $\hat{\alpha}(v)$ is the minimum value such that
    $\delta_{T,\hat{\alpha}}$ is nearest to $\delta$ in the $l^\infty$-norm,
    then we say that $\hat{u}$ is obtained from from $u$ by \emph{sliding $v$ all the way down}.
\end{defn}

\begin{ex}
    Let $\delta$ be the dissimilarity map shown on the left in Figure \ref{fig:multipleClosest}
    and consider the $l^\infty$-nearest ultrametrics $u_1,u_2,$ and $u_3$ shown in Figure \ref{fig:algTrees}.
    Note that $u_2$ is obtained from $u_1$ by sliding the node with weight $7$ all the way down,
    and $u_3$ is obtained from $u_1$ by sliding the node with weight $5$ all the way down.
\end{ex}



\begin{thm}\label{thm:treesAlgorithm}
    Let $\delta: X \times X \rightarrow \rr$ be a dissimilarity map.
    Let $S_0 = \{\delta_m\}$, and for each $i \ge 1$ define
    $S_i$ to be the set of ultrametrics obtained from some $u \in S_{i-1}$
    by sliding a mobile internal node of a resolution of the topology of $u$ all the way down.
    Then
    \begin{enumerate}
        \item $\bigcup_i S_i$ is a finite set, and
        \item\label{item:treesTropConvexHull} the tropical convex hull of $\bigcup_i S_i$ is the set of ultrametrics $l^\infty$-nearest to $\delta$, and
        \item\label{item:treesOneMobileNode} every vertex of this tropical polytope has at most one mobile internal node.
    \end{enumerate}
\end{thm}

Theorem \ref{thm:treesAlgorithm} is a special case of Theorem \ref{thm:verticesAlgorithm},
which will be proven later.
We now illustrate Theorem \ref{thm:treesAlgorithm} on an example.

\begin{ex}\label{ex:algTrees}
    Let $\delta$ be the dissimilarity map given in Figure \ref{fig:multipleClosest} on the left.
    We will make reference to ultrametrics $u_1,\dots,u_5$ which are shown in Figure \ref{fig:algTrees}.
    Using Theorem \ref{thm:chepoi},
    we can see that $\delta_m = u_1$.
    Let $v_1$ be the internal node of $u_1$'s topology with weight $5$.
    Then $v_1$ is mobile and sliding it all the way down yields $u_3$.
    Let $v_2$ be the internal node of $u_1$'s topology with weight $7$.
    Then $v_2$ is mobile and sliding it all the way down yields $u_2$.
    The topology of $u_4$ is a resolution of the topology of $u_2$.
    Letting $v_3$ be the internal node of $u_4$'s topology with weight $1$,
    we can see that $u_4$ is obtained from $u_2$ by sliding $v_3$ all the way down.
    The topology of $u_5$ is also a resolution of the topology of $u_2$.
    Letting $v_4$ be the internal node of $u_5$'s topology with weight $-1$,
    we can see that $u_5$ is obtained from $u_2$ by sliding $v_4$ all the way down.
    Beyond $v_3$ and $v_4$, no internal nodes of any resolution of the topology of $u_2$ are mobile.
    The only mobile node of $u_3$ is the node labeled $7$;
    denote this $v_5$.
    Then sliding $v_5$ all the way down gives us $u_5$ once again.
    
    Using the notation of Theorem \ref{thm:treesAlgorithm},
    we have $S_0 = \{u_1\}$, $S_1 = \{u_2,u_3\}$ and $S_2 = \{u_4,u_5\}$.
    Note that no internal nodes of $u_4$ and $u_5$ are mobile.
    Hence $S_i$ is empty for all $i \ge 3$.
    Since $u_1$ and $u_2$ each have two mobile internal nodes,
    Theorem \ref{thm:treesAlgorithm} implies
    that the set of ultrametrics $l^\infty$-nearest to $\delta$
    is the tropical convex hull of $\{u_3,u_4,u_5\}$.
    This tropical polytope is contained in the three-dimensional affine subspace 
    $\{\tilde{\delta} \in \rr^{\binom{[4]}{2}}: \tilde{\delta}(1,4) = \tilde{\delta}(2,4) 
    = \tilde{\delta}(3,4) = 9\}\subset\rr^{\binom{[4]}{2}}$.
    Therefore, we can visualize it as in Figure \ref{fig:treeTropicalPolytope}.
    
    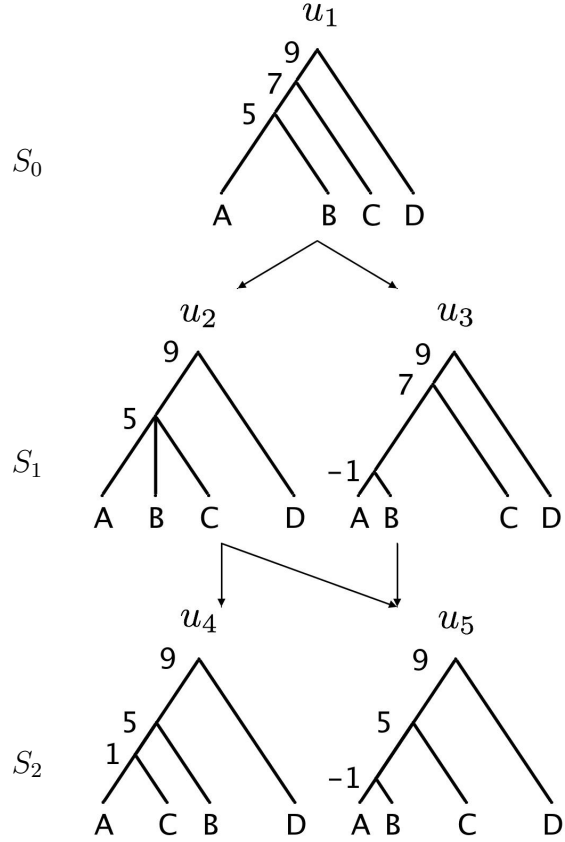
\begin{figure}[h]\centering
        \begin{subfigure}{0.38\textwidth}\centering
            \begin{tikzpicture}
                \node at (-.8,3/4){$u_1=$};
                \node (a) at (0,-0.2){$A$};
                \node (b) at (1,-0.2){$B$};
                \node (c) at (2,-0.2){$C$};
                \node (d) at (3,-0.2){$D$};
                \node (ab) at (1/2-0.2,1/2+0.2){$5$};
                \node (abc) at (1-0.2,1+0.2){$7$};
                \node (abcd) at (3/2-0.2,3/2+0.2){$9$};
                \draw (0,0) -- (1/2,1/2);
                \draw (1,0) -- (1/2,1/2);
                \draw (1/2,1/2) -- (1,1);
                \draw (2,0) -- (1,1);
                \draw (3,0) -- (3/2,3/2);
                \draw (1/2,1/2) -- (3/2,3/2);
            \end{tikzpicture}
        \end{subfigure}
        \begin{subfigure}{0.38\textwidth}\centering
            \begin{tikzpicture}
                \node at (-.8,3/4){$u_2=$};
                \node (a) at (0,-0.2){$A$};
                \node (b) at (1,-0.2){$B$};
                \node (c) at (2,-0.2){$C$};
                \node (d) at (3,-0.2){$D$};
                \node (abc) at (1-0.2,1+0.2){$5$};
                \node (abcd) at (3/2-0.2,3/2+0.2){$9$};
                \draw (0,0) -- (1,1);
                \draw (1,0) -- (1,1);
                \draw (2,0) -- (1,1);
                \draw (3,0) -- (3/2,3/2);
                \draw (1/2,1/2) -- (3/2,3/2);
            \end{tikzpicture}
        \end{subfigure}
        \ \\ \ \\
        \begin{subfigure}{0.32\textwidth}\centering
            \begin{tikzpicture}
                \node at (-.8,3/4){$u_3=$};
                \node (a) at (0,-0.2){$A$};
                \node (b) at (1,-0.2){$B$};
                \node (c) at (2,-0.2){$C$};
                \node (d) at (3,-0.2){$D$};
                \node (ab) at (1/2-0.2,1/2+0.2){$-1$};
                \node (abc) at (1-0.2,1+0.2){$7$};
                \node (abcd) at (3/2-0.2,3/2+0.2){$9$};
                \draw (0,0) -- (1/2,1/2);
                \draw (1,0) -- (1/2,1/2);
                \draw (1/2,1/2) -- (1,1);
                \draw (2,0) -- (1,1);
                \draw (3,0) -- (3/2,3/2);
                \draw (1/2,1/2) -- (3/2,3/2);
            \end{tikzpicture}
        \end{subfigure}
        \begin{subfigure}{0.32\textwidth}\centering
            \begin{tikzpicture}
                \node at (-.8,3/4){$u_4=$};
                \node (a) at (0,-0.2){$A$};
                \node (b) at (1,-0.2){$C$};
                \node (c) at (2,-0.2){$B$};
                \node (d) at (3,-0.2){$D$};
                \node (ab) at (1/2-0.2,1/2+0.2){$1$};
                \node (abc) at (1-0.2,1+0.2){$5$};
                \node (abcd) at (3/2-0.2,3/2+0.2){$9$};
                \draw (0,0) -- (1/2,1/2);
                \draw (1,0) -- (1/2,1/2);
                \draw (1/2,1/2) -- (1,1);
                \draw (2,0) -- (1,1);
                \draw (3,0) -- (3/2,3/2);
                \draw (1/2,1/2) -- (3/2,3/2);
            \end{tikzpicture}
        \end{subfigure}
        \begin{subfigure}{0.32\textwidth}\centering
            \begin{tikzpicture}
                \node at (-.8,3/4){$u_5=$};
                \node (a) at (0,-0.2){$A$};
                \node (b) at (1,-0.2){$B$};
                \node (c) at (2,-0.2){$C$};
                \node (d) at (3,-0.2){$D$};
                \node (ab) at (1/2-0.2,1/2+0.2){$-1$};
                \node (abc) at (1-0.2,1+0.2){$5$};
                \node (abcd) at (3/2-0.2,3/2+0.2){$9$};
                \draw (0,0) -- (1/2,1/2);
                \draw (1,0) -- (1/2,1/2);
                \draw (1/2,1/2) -- (1,1);
                \draw (2,0) -- (1,1);
                \draw (3,0) -- (3/2,3/2);
                \draw (1/2,1/2) -- (3/2,3/2);
            \end{tikzpicture}
        \end{subfigure}
        \color{black}
        \caption{By Theorem \ref{thm:treesAlgorithm},
        the tropical convex hull of the ultrametrics above
        is the set of ultrametrics $l^\infty$-nearest
        to the dissimilarity map given on
        the left side of Figure \ref{fig:multipleClosest}.}\label{fig:algTrees}
    \end{figure}
\end{ex}

\begin{figure}
    \includegraphics[scale=0.1]{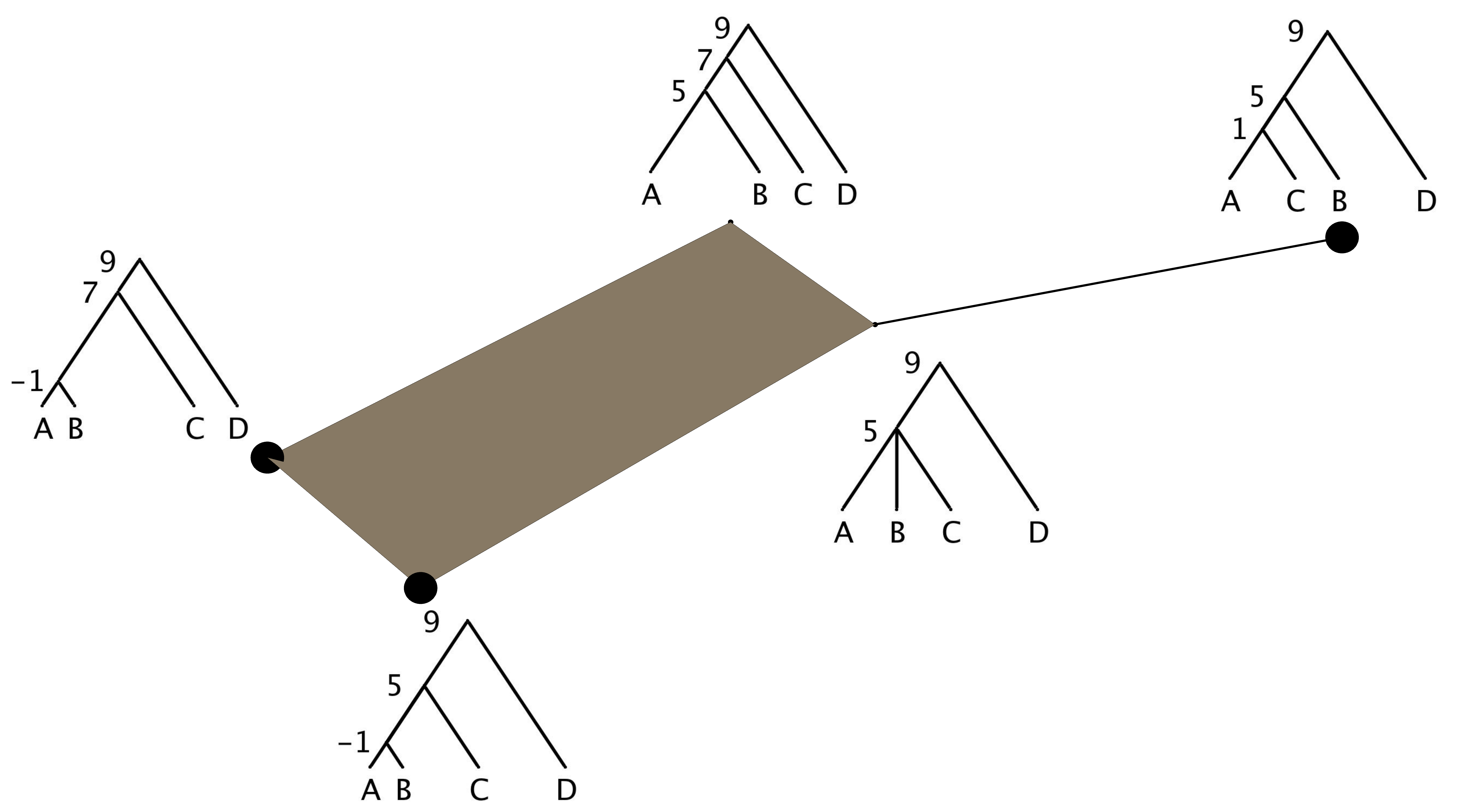}
    \caption{The tropical polytope consisting of ultrametrics
    that are $l^\infty$-nearest to $\delta$.
    The large points are the tropical vertices.}\label{fig:treeTropicalPolytope}
\end{figure}

{
Theorem \ref{thm:treesAlgorithm} implies that the elements of $\bigcup_i S_i$
that have at most one mobile internal vertex are a superset of the 
vertex set of the tropical polytope consisting of the ultrametrics $l^\infty$-nearest to a given dissimilarity map.
A recent preprint of Luyan Yu shows that this containment can be strict
for dissimilarity maps with at least four elements \cite{yu2019extreme}.
A complete characterization of the vertices of this tropical polytope is still open.

We now describe a polynomial-time checkable condition that is equivalent to the condition that all ultrametrics $l^\infty$-nearest
to a given dissimilarity map have the same topology.

\begin{thm}\label{thm:sameTopologyTest}
    All ultrametrics $l^\infty$-nearest to a given dissimilarity map $\delta$ on $n$ elements
    have the same topology if and only if the ultrametrics in $S_0 \cup S_1$ from Theorem \ref{thm:treesAlgorithm}
    all have the same topology.
    This condition can be checked in $O(n^2)$ time.
\end{thm}
\begin{proof}
    The second claim follows from the fact that Chepoi and Fichet's algorithm in Theorem \ref{thm:chepoi}
    runs in $O(n^2)$ time (c.f. \cite{chepoi}),
    and that $\delta_m$ has at most $n-1$ internal vertices.

    If we were to replace $S_0 \cup S_1$ with $\bigcup_i S_i$ in the statement of the theorem,
    then it would immediately follow from Proposition \ref{prop:multipletopologies}.
    So it suffices to show that if all trees in $S_0 \cup S_1$ have the same topology,
    then all trees in $\bigcup_i S_i$ do as well.

    Assume that all ultrametrics in $S_0 \cup S_1$ have the same topology $T$.
    For the sake of contradiction, let $i \ge 2$ be minimal such that there exists $u \in S_i$
    such that the topology of $u$ is not $T$.
    This means that there exists some $u' \in S_{i-1}$ with topology $T$
    such that an internal node $v$ of $T$ is mobile in $u'$
    and sliding $v$ all the way down in $u'$ yields $u$.
    Let $\alpha_m,\alpha,\alpha':T^\circ \rightarrow \rr$
    be internal edge weightings of $T$
    expressing $\delta_m,u,$ and $u'$ respectively
    (i.e., $\delta_m = \delta_{T,\alpha_m}$,
    $u = \delta_{T,\alpha}$, and $u' = \delta_{T,\alpha'}$).
    Since $u$ is obtained from $u'$ by sliding $v$ all the way down,
    $\alpha(y) = \alpha'(y)$ unless $y = v$, in which case
    $\alpha(v) = \max\{\alpha'(y): y \in \des_T(v)\}$.
    Since $\max\{\alpha_m(y): y \in \des_T(v)\} \ge \max\{\alpha'(y): y \in \des_T(v)\}$,
    $\delta_m$ and $u'$ both have topology $T$,
    and all internal nodes of $T$ that are mobile for $u'$ are also mobile for $\delta_m$,
    we can slide $v$ all the way down in $\delta_m$ to get an element of $S_1$
    with the topology of $u$, contradicting that all elements of $S_1$ have topology $T$.
\end{proof}
}

\section{Bergman fans and nested sets}\label{sec:nestedSets}
 The goal of this section is to generalize the notion of tree topology
for ultrametrics to elements of Bergman fans of arbitrary matroids.
Nested sets of matroids, as described in \cite{feichtner-sturmfels2005},
will play the role of rooted trees in this more general context.
Familiarity with matroid connectivity is assumed; for this we refer the reader to \cite[Chapter~4]{oxley2011}.
We begin by defining the Bergman fan of a matroid.
Equivalent cryptomorphic definitions exist.
The one we provide is due to Ardila (see \cite[Proposition 2]{ardila2004}).

\begin{defn}\label{defn:bergmanFan}
    Let $\calm$ be a matroid on ground set $E$.
    A vector $w \in \rr^E$ is said to be an \emph{$\calm$-ultrametric} if
    for each circuit $C$ of $\calm$,
    the cardinality of $\{x \in C: w_x = \max_{y\in C} w_y\}$
    is at least two.
    The set of $\calm$-ultrametrics, denoted $\berg(\calm)$,
    is called the \emph{Bergman fan of~$\calm$}.
\end{defn}

\color{black}
As the name suggests,
$\calm$-ultrametrics generalize the ultrametrics discussed
in Section~\ref{sec:phylogenetics}.
In particular, letting $K_n$ denote the complete graph on $n$ vertices
and $\calm(G)$ denote the matroid underlying a graph $G$,
the following theorem of Ardila and Klivans
tells us that ultrametrics are $\calm(K_n)$-ultrametrics.

\begin{thm}[\cite{ardila-klivans2006}, Theorem 3]
\label{thm:ardilaUltrametric}
    A dissimilarity map on the set $\{1,\dots,n\}$ is an ultrametric
    if and only if it is an $\calm(K_n)$-ultrametric.
\end{thm}

We would like to generalize Theorem \ref{thm:treesAlgorithm},
i.e.~describe a generating set of the tropical polytope consisting of the $\calm$-ultrametrics
that are $l^\infty$-nearest to a given $x \in \rr^E$.
To do this, we need to generalize the notion of tree topology for arbitrary $\calm$-ultrametrics.
Definition \ref{defn:nestedSets} below provides the desired generalization.
It is essentially the special case of Definition 3.2 in \cite{feichtner-sturmfels2005} 
where the required lattice is the lattice of flats of a connected matroid $\calm$
and the required building set is the set of connected flats of $\calm$.

\begin{defn}\label{defn:nestedSets}
    Given a connected matroid $\calm$ on ground set $E$,
    a \emph{nested set of $\calm$} is a set
    $\cals$ of connected nonempty flats of $\calm$
    such that $E \in \cals$ and whenever $F_1,\dots,F_k \in \cals$ are pairwise incomparable with respect to the containment order,
    the closure of $F_1 \cup \dots \cup F_k$ is disconnected.
    If $\calm$ is disconnected with connected components $\calm_1,\dots,\calm_k$,
    then a nested set of $\calm$ is the union of nested sets $\cals_1,\dots,\cals_k$
    of $\calm_1,\dots,\calm_k$.
\end{defn}

\begin{ex}\label{ex:nestedSets}
    Let $\calm$ be the uniform matroid of rank three on ground set $\{a,b,c,d\}$.
    The nested sets of $\calm$ are the sets of any of the following forms
    \begin{align*}
        \{\{a,b,c,d\}\} \quad \{\{x\},\{a,b,c,d\}\} \quad \{\{x\},\{y\},\{a,b,c,d\}\}
    \end{align*}
    where $x,y \in \{a,b,c,d\}$.
    If $\caln$ is the uniform matroid of rank two on ground set $\{e,f,g\}$,
    then the nested sets of the direct sum $\calm \oplus \caln$
    are sets of the form
    \[
        S\cup \{\{e,f,g\}\} \quad S \cup \{\{x\},\{e,f,g\}\}
    \]
    where $S$ is a nested set of $\calm$ and $x \in \{e,f,g\}$.
\end{ex}
\color{black}

We remind the reader that ${\bf 1}$ denotes the all-ones vector.

\begin{defn}\label{defn:nestedSetComplex}
	Let $\calm$ be a connected matroid on ground set $E$
	and let $\cals$ be a nested set of $\calm$.
    For each $F \in \cals$, let $v_F \in \rr^E$
    denote $-1$ times the characteristic vector of $F$.
    Define $K_\cals$ to be the cone spanned by the $v_F$
    and $\pm {\bf 1}$.
    The \emph{nested set fan of $\calm$}, denoted $\nest(\calm)$ is the polyhedral fan
    consisting of all the polyhedral cones $K_\cals$ as $\cals$
    ranges over all nested sets of $\calm$.
    When $\calm$ is disconnected, we define its nested set fan 
    to be the cartesian product of the nested set fans of its connected components.
\end{defn}

Note that $\nest(\calm)$ is indeed a polyhedral fan
since $K_\cals$ is simplicial, and $K_{\cals} \cap K_{\cals'} = K_{\cals \cap \cals'}$.
Also note that the lineality space of $\nest(\calm)$ is spanned by
the characteristic vectors of the connected components of $\calm$.

Definition \ref{defn:nestedSets} is slightly more restrictive than Definition 3.2 of \cite{feichtner-sturmfels2005}.
Namely, a nested set in the sense of \cite{feichtner-sturmfels2005} does not require
that each connected component of a matroid be present,
nor that the entire ground set of a disconnected matroid not be present.
{ For example, using $\calm$ and $\caln$ as in Example \ref{ex:nestedSets},
Definition 3.2 of \cite{feichtner-sturmfels2005}
would allow us to remove $\{a,b,c,d\}$ from any nested set of $\calm$,
or add $\{a,b,c,d,e,f,g\}$ to any nested set of $\calm \oplus \caln$.}
However, this is not an issue because these differences in definitions do not affect the nested set fan.
{ Under the less restrictive definition,
if $E \in S$ for some nested set $S$,
then $K_S = K_{S \setminus\{E\}}$.
We use this more restrictive definition to avoid this ambiguity when indexing cones of $\nest(\calm)$.}

\begin{prop}\label{prop:nestedRefinesBergman}
    The nested set fan $\nest(\calm)$ is a refinement of the Bergman fan $\berg(\calm)$.
\end{prop}
\begin{proof}
    When $\calm$ is connected, this follows from Theorem 4.1 in \cite{feichtner-sturmfels2005}.
    The rest of the proposition follows by noting that the Bergman fan of a disconnected matroid
    is the cartesian product of the Bergman fans of its connected components.
\end{proof}

{ In light of Proposition \ref{prop:nestedRefinesBergman},
we can make the following definition.}

\begin{defn}\label{defn:bergmanFanTopology}
    Let $w$ be an $\calm$-ultrametric.
    { Let $\calt(w)$ denote the unique nested set of of $\calm$
    such that $w$ lies in the relative interior of $K_{\calt(w)}$.
    We call $\calt(w)$ the \emph{topology of $w$}.}
\end{defn}

Definition \ref{defn:bergmanFanTopology} might be unsettling to some readers
since it appears to have nothing to do with topology in the usual sense.
We use it because, as Proposition \ref{prop:generalizationWorks} below shows,
it generalizes the notion of tree topology of an ultrametric
in the phylogenetics sense.

\begin{prop}[{\cite[Remark 5.4]{feichtner-sturmfels2005}}]\label{prop:generalizationWorks}
    Let $w,u$ be $\calm(K_n)$-ultrametrics.
    Then the tree topologies of $w,u$ are equal if and only if $\calt(w) = \calt(u)$.
\end{prop}

The following proposition tells us that topology of $\calm$-ultrametrics is
well-behaved with respect to tropical convexity.

\begin{prop}\label{prop:topologyconvex}
    The set of $\calm$-ultrametrics that have a particular topology $\cals$ is tropically convex.
\end{prop}
\begin{proof}
	The lineality space of $\nest(\calm)$ contains ${\bf 1}$
	so the topology of an $\calm$-ultrametric is preserved under tropical scalar multiplication.
    We now show that topology is preserved under tropical sums.
    To this end, let $u,w$ be $\calm$-ultrametrics that lie in the relative interior of
    the same cone $K_\cals$.
    Modulo the lineality space of $\nest(\calm)$,
    $u = \sum_{F \in \cals} \lambda_F^u v_F$ and $w = \sum_{F \in \cals} \lambda_F^w v_F$
    where the sums are taken over the flats in $\cals$ that are not connected components of $\calm$,
    and $\lambda_F^u,\lambda_F^w$ are all strictly positive.
    Then $(u \oplus w) = \sum_F (\max\{\lambda_F^u, \lambda_F^w\}) v_F$.
    So $(u \oplus w)$ also lies in the relative interior of $K_\cals$
    and so $\calt(u \oplus w) = \cals$.
\end{proof}

Lemma \ref{lem:hierarchy} below implies that the Hasse diagram of the containment partial ordering
on a nested set of a matroid $\calm$ is a forest with a tree for each connected component of $\calm$.
Proposition \ref{prop:displayOnTopology} implies that each $\calm$-ultrametric can be displayed on
this forest in the same way that an ultrametric can be displayed on its tree topology.

\begin{lemma}\label{lem:hierarchy}
    Let $\cals$ be a nested set of a matroid $\calm$.
    Then for any pair $F,G \in \cals$,
    $F\subseteq G$ or $G\subseteq F$ or $G \cap F = \emptyset$.
\end{lemma}
\begin{proof}
    Assume $F$ and $G$ are connected flats of $\calm$
    and that $F \cap G \neq \emptyset$.
    We will show that the closure $K$ of $F \cup G$ is connected.
    It will then follow from the definition of a nested set that either $F \subseteq G$
    or $G \subseteq F$.
    Let $\sim$ be the relation on $K$ where $a\sim b$
    if and only if there exists a circuit $C \subseteq K$
    containing both $a$ and $b$.
    It suffices to show that there is only one equivalence class of $K$ under $\sim$ \cite[Chapter~4.1]{oxley2011}.
    Both $F$ and $G$ are connected,
    so each must lie entirely within one equivalence class.
    Moreover, their intersection is nontrivial so $F \cup G$ lies
    in a single equivalence class.
    Since $K$ is the closure of $F \cup G$,
    each $e \in K\setminus(F \cup G)$ must also lie in this equivalence class.
\end{proof}

Note that Lemma \ref{lem:hierarchy} implies that if $\cals$ is a nested set of a matroid $\calm$,
then for each $e$ in the ground set of $\calm$,
there is a unique minimal flat in $\cals$ that contains $e$.

\begin{defn}\label{defn:multrametricFromTopology}
	Let $\calm$ be a matroid on ground set $E$ and let $\cals$ be a nested set of $\calm$.
	A function $\alpha: \cals \rightarrow \rr$ is said to be \emph{compatible with $\cals$} 
    if $F \subseteq G$ implies $\alpha(F) \le \alpha(G)$ for all $F,G \in \cals$.
    For $\alpha$ compatible with $\cals$,
	define $w^{\cals,\alpha} \in \rr^E$ by $w^{\cals,\alpha}_e = \alpha(F)$
	where $F$ is the minimal flat in $\cals$ that contains $e$.
	If $w = w^{\cals,\alpha}$,
	then we call the pair $(\cals,\alpha)$
	a \emph{nested set representation of $w$ on $\cals$}.
\end{defn}

\begin{prop}\label{prop:displayOnTopology}
    Let $\calm$ and $\cals$ be as in Definition \ref{defn:multrametricFromTopology} and let $\alpha: \cals \rightarrow \rr$
    be compatible with $\cals$.
    Then $w^{\cals,\alpha}$ is an $\calm$-ultrametric.
	Every $\calm$-ultrametric $w$
	has a unique nested set representation $w = w^{\calt(w),\alpha}$ on its topology.
\end{prop}
\begin{proof}
    It is sufficient to prove the proposition in the case where $\calm$ is connected,
    so assume $\calm$ is connected.
    We first show that $w^{\cals,\alpha}$ is indeed an $\calm$-ultrametric.
    Define $\lambda_E := -\alpha(E)$
    and for each $F \in \cals\setminus E$, define
    $\lambda_F := -\alpha(F) + \alpha(G)$
    where $G$ is the minimal element of $\cals$ strictly containing $F$ (Lemma \ref{lem:hierarchy} implies that
    a unique such $G$ exists).
    For each $F \in \cals$, let $v_F$ be as in Definition \ref{defn:nestedSetComplex}.
    Then $w^{\cals,\alpha} = \sum_{F \in \cals} \lambda_F v_F$.
    Since $\alpha$ is compatible with $\cals$,
    $F \neq E$ implies that $\lambda_F$ is nonnegative.
    This shows that $w^{\cals,\alpha}$ is in the nested set fan.
    Proposition \ref{prop:nestedRefinesBergman} then implies that $w^{\cals,\alpha}$
    is a $\calm$-ultrametric.

    Now let $w$ be an arbitrary $\calm$-ultrametric.
    By Proposition \ref{prop:nestedRefinesBergman} and Definition \ref{defn:bergmanFanTopology},
    $w = \sum_{F \in \calt(w)} \lambda_F v_F$
    for some choice of coefficients $\lambda_F$
    satisfying $\lambda_F > 0$ when $F \neq E$.
    Set $\alpha(E) := -\lambda_E$, and for each $F \in \calt(w) \setminus \{E\}$
    inductively set $\alpha(F) := -\lambda_F + \alpha(G)$
    where $G$ is the minimal element of $\calt(w)$ containing $F$.
    Note that $\alpha(F) < \alpha(F')$ whenever $F \subsetneq F'$ and that $w = w^{\calt(w),\alpha}$.
    Uniqueness of $\alpha$ follows from the fact that this map from the $\lambda_F$'s
    to the $\alpha(F)$'s is invertible and that $\{v_F: F \in \calt(w)\}$
    is a linearly independent set.
\end{proof}

Proposition \ref{prop:displayOnTopology} gives us a way to display an $\calm$-ultrametric
that generalizes the way we can display an ultrametric on its tree topology.
Namely, if $w$ is an $\calm$-ultrametric and $\alpha: \calt(w) \rightarrow \rr$
is such that $w = w^{\calt(w),\alpha}$,
we can specify $w$ by drawing the Hasse diagram for $\calt(w)$
(which is a forest by Lemma \ref{lem:hierarchy})
and labeling each $F \in \calt(w)$ with $\alpha(F)$.
We now show this in an example.

\begin{ex}\label{ex:displayOnTopology}
	The left side of Figure \ref{fig:multrametricOnTopology}
	displays a $\calm(G)$-ultrametric $w$ as an edge weighting of the graph $G$.
	On its right is $\calt(w)$ where each flat $F \in \calt(w)$
	is labeled by $\alpha(F)$ where $\alpha: \calt(w) \rightarrow \rr$
	satisfies $w = w^{\calt(w),\alpha}$.
	Since the graph $G$ is not biconnected,
	the matroid $\calm(G)$ is disconnected and so $\calt(w)$ is disconnected.
    \begin{figure}\centering
        \begin{subfigure}{0.49\textwidth}
            \begin{tikzpicture}
                \vertex (a) at (0,0)[label=left:$a$]{};
                \vertex (b) at (1,1)[label=above:$b$]{};
                \vertex (c) at (1,-1)[label=below:$c$]{};
                \vertex (d) at (2,0)[label=above:$d$]{};
                \vertex (e) at (3,1)[label=above:$e$]{};
                \vertex (f) at (3,-1)[label=below:$f$]{};
                \vertex (g) at (4,0)[label=right:$g$]{};
                \path
                    (a) edge node[above,pos=.4]{$1$} (b) 
                    (a) edge node[below,pos=.3]{$3$} (c)
                    (b) edge node[above,pos=.6]{ $2$} (d)
                    (c) edge node[below,pos=.7]{$3$} (d)
                    (a) edge node[above]{$2$} (d)
                    (d) edge node[above,pos=.4]{$3$} (e)
                    (d) edge node[below,pos=.3]{$2$}  (f)
                    (e) edge node[right]{$3$} (f)
                    (e) edge node[above,pos=.6]{$3$} (g)
                    (f) edge node[right,pos=.4]{$1$} (g)
                ;
            \end{tikzpicture}
        \end{subfigure}
        \begin{subfigure}{0.49\textwidth}
        	\hspace{-5ex}
        	\begin{tikzpicture}
        	    \node (a) at (0,2){$3 \ \{ab,ac,ad,bd,cd\}$};
        	    \node (b) at (0,1){$2 \ \{ab,ad,bd\}$};
        	    \node (c) at (0,0){$1 \ \{ab\}$};
        	    \node (d) at (4.5,2){$3 \ \{de,df,ef,eg,fg\}$};
        	    \node (e) at (3.5,0.5){$2 \ \{df\}$};
        	    \node (f) at (5.5,0.5){$1 \ \{fg\}$};
        	    \draw (a) -- (b);
        	    \draw (b) -- (c);
        	    \draw (d) -- (e);
        	    \draw (d) -- (f);
        	\end{tikzpicture}
        \end{subfigure}
        \caption{An $\calm(G)$-ultrametric $w$,
        displayed as an edge-weighting of $G$
        and using the $\alpha: \calt(w) \rightarrow \rr$ as described
        in Proposition \ref{prop:displayOnTopology}.}\label{fig:multrametricOnTopology}
    \end{figure}
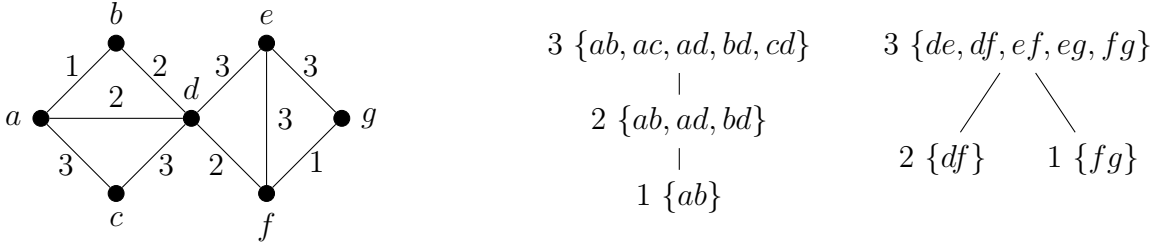
\end{ex}

We now generalize the concepts of polytomy and resolution from rooted trees
representing ultrametrics to nested sets representing $\calm$-ultrametrics.

\begin{defn}\label{defn:polytomyAndResolution}
	Let $\calm$ be a matroid on ground set $E$ and let $\cals$ be a nested set of $\calm$.
	A \emph{polytomy of $\cals$} is an element $F \in \cals$ such that
	$\rank(F / \bigcup_G G) > 1$ where the union is taken over all $G \in \mathcal{S}$
	such that $G \subsetneq F$.
	A \emph{resolution} of $\cals$ is another nested set $\cals'$
	without polytomies such that $\cals \subseteq \cals'$.
\end{defn}

If $\calt(w)$ has a polytomy,
then the nested set representation of $w$ is not unique.
In particular, $w$ can be represented on any nested set $\cals$
that is a resolution of $\calt(w)$.

\begin{ex}\label{ex:polytomyAndResolution}
	On the left side of Figure \ref{fig:polytomyAndResolution},
	we see a nested set $\cals_1$ of the matroid $\calm(K_4)$ underlying the complete graph
	on vertex set $\{a,b,c,d\}$.
	Since $\calm(K_4)/\{ab\}$ is a matroid of rank $2$,
	the set $\{ab,ac,ad,bc,bd,cd\}$ is a polytomy of $\cals_1$.
	To its right are the two possible resolutions $\cals_2$ and $\cals_3$.
	Each $\cals_i$ is shown with a compatible $\alpha_i : \cals_i \rightarrow \rr$,
	thus giving us the $\calm(K_4)$-ultrametrics $w^{\cals_i,\alpha_i}$.
	Note that $w^{\cals_1,\alpha_1} = w^{\cals_2,\alpha_2} = w^{\cals_3,\alpha_3}$
	and that the topology of this $\calm(K_4)$-ultrametric is $\cals_1$.
	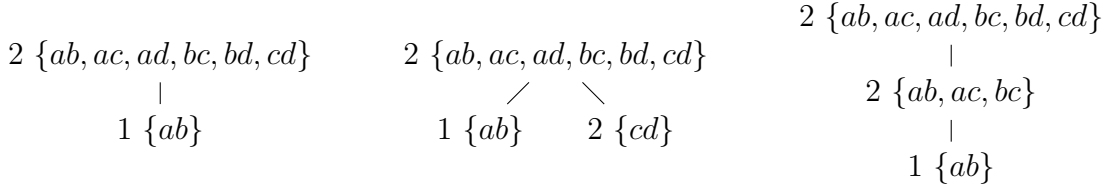
\begin{figure}[h]
		\begin{subfigure}{0.32\textwidth}
	        \begin{tikzpicture}
                \node at (0,2){};
                \node at (0,0){};
	        	\node (a) at (0,1.5){$2 \ \{ab,ac,ad,bc,bd,cd\}$};
        	    \node (b) at (0,0.5){$1 \ \{ab\}$};
        	    \draw (a) -- (b);
                \node at (0,-1.225){$\mathcal{S}_1$};
	        \end{tikzpicture}
	    \end{subfigure}
		\begin{subfigure}{0.32\textwidth}
	        \begin{tikzpicture}
                \node at (0,2){};
                \node at (0,0){};
	        	\node (a) at (1,1.5){$2 \ \{ab,ac,ad,bc,bd,cd\}$};
	        	\node (b) at (2,0.5){$2 \ \{cd\}$};
        	    \node (c) at (0,0.5){$1 \ \{ab\}$};
        	    \draw (a) -- (b);
        	    \draw (a) -- (c);
                \node at (1,-1.225){$\mathcal{S}_2$};
	        \end{tikzpicture}
	    \end{subfigure}
	    \begin{subfigure}{0.32\textwidth}
	        \begin{tikzpicture}
	        	\node (a) at (0,2){$2 \ \{ab,ac,ad,bc,bd,cd\}$};
        	    \node (b) at (0,1){$2 \ \{ab,ac,bc\}$};
        	    \node (c) at (0,0){$1 \ \{ab\}$};
                \node at (0,1.5){};
                \node at (0,0.5){};
        	    \draw (a) -- (b) -- (c);
                \node at (0,-1){$\mathcal{S}_3$};
	        \end{tikzpicture}
	    \end{subfigure}
	\caption{
		A nested set $\cals_1$ of the complete graph on vertex set $\{a,b,c,d\}$
		with a polytomy	and its two resolutions $\cals_2$ and $\cals_3$.
		The weightings on each nested set all give rise to the
		same $\calm(K_4)$-ultrametric.
	}\label{fig:polytomyAndResolution}
	\end{figure}
\end{ex}


\section{L-infinity optimization to Bergman fans of matroids}\label{sec:bergmanOptimization}
The first important result of this section is Proposition \ref{prop:tropicalPolytope},
which says that the subset of a Bergman fan $\berg(\calm) \subseteq \rr^E$ consisting
of all points $l^\infty$-nearest to a given $x \in \rr^E$ is a tropical polytope.
The main result of this section is Theorem \ref{thm:verticesAlgorithm},
which describes a generating set of this tropical polytope.
In light of Proposition \ref{prop:completion},
Theorem \ref{thm:verticesAlgorithm} is applicable for ultrametric reconstruction
in cases where the data consists only of a subset of all pairwise distances.
We begin by recalling a result of Ardila,
establishing a connection between ultrametric reconstruction and tropical convexity.

\begin{prop}[\cite{ardila2004}, Proposition 4.1]\label{prop:linearpolytope}
    The Bergman fan $\berg(\calm)$ is a tropical polyhedral cone.
\end{prop}

We introduce some notation.
Given points $x,y \in \rr^E$ and a set $S \subseteq \rr^E$,
we denote the $l^\infty$-distance between $x$ and $y$ by $d(x,y)$
and $\inf_{y \in S}d(x,y)$ by $d(x,S)$.
Given some $x \in \rr^E$, we define the subset of $\berg(\calm)$
consisting of the $\calm$-ultrametrics that are $l^\infty$-nearest to $x$
by $C(x,\berg(\calm))$.
That is, $C(x,\berg(\calm)) = \{w \in \berg(\calm): d(x,w) = d(x,\berg(\calm))\}$.
The next proposition says that this set is a tropical polytope.

\begin{prop}\label{prop:tropicalPolytope}
    If $\calm$ is a matroid on ground set $E$ and $x \in \rr^E$,
    then the subset of the Bergman fan of $\calm$ consisting
    of elements $l^\infty$-nearest to $x$
    is a tropical polytope.
\end{prop}
\begin{proof}
    Let $C$ denote the cube of side-length $d(x,\berg(\calm))$ centered at $x$.
    Therefore we can express $C(x,\berg(\calm)) = \berg(\calm) \cap C$.
    Proposition \ref{prop:linearpolytope} tells us that $\berg(\calm)$ is a tropical polyhedron
    and $C$ is clearly a tropical polytope.
    Their intersection is again a tropical polyhedron.
    Since it is bounded it is
    by definition a tropical polytope.
\end{proof}

Much of the remainder of this section is devoted to describing the set of
tropical vertices of $C(x,\berg(\calm))$.
Now we recall the concept of a subdominant $\calm$-ultrametric,
the existence of which was proven by Ardila in \cite{ardila2004}.

\begin{defn}[\cite{ardila2004}]\label{defn:subdominant}
    Let $\calm$ be a matroid on ground set $E$ and let $x \in \rr^E$.
    Let $x^\calm$ denote the unique coordinate-wise maximum $\calm$-ultrametric which is coordinate-wise at most $x$.
    We call $x^\calm$ the \emph{subdominant $\calm$-ultrametric of $x$}.
\end{defn}

Given some $x \in \rr^E$, Ardila shows how the first three steps of the algorithm from
Theorem \ref{thm:chepoi} can be extended to compute the subdominant $\calm$-ultrametric of $x$.
Then the subdominant ultrametric can be shifted to obtain an $l^\infty$-nearest ultrametric
that is coordinate-wise maximal among all $l^\infty$-nearest ultrametrics.

\begin{lemma}\label{lem:maximalClosestMUltrametric}
    Let $\calm$ be a matroid on ground set $E$, $x \in \rr^E$, and $\delta = \frac{1}{2}d(x,x^\calm)$.
    Then
    \begin{enumerate}
        \item\label{item:distance} The $l^\infty$-distance from $x$ to $\berg(\calm)$ is $\delta$,
        \item\label{item:canonicalclosest} $x^\calm + \delta \cdot {\bf 1}$ is an $\calm$-ultrametric, $l^\infty$-nearest to $x$,
        \item\label{item:maximal} $x^\calm + \delta \cdot {\bf 1}$ is maximal among $\calm$-ultrametrics
        $l^\infty$-nearest to $x$.
    \end{enumerate}
\end{lemma}
\begin{proof}
    The existence of $x^\calm + \delta \cdot {\bf 1}$ shows that $d(x,\berg(\calm)) \le \delta$.
    Suppose there exists $w \in \berg(\calm)$ such that $d(x,w) < \delta$.
    Then $w - d(x,w)\cdot{\bf 1}$ is coordinate-wise at most $x$.
    There exists
    $e \in E$ such that $x_e - x^\calm_e = 2 \delta$
    and so $x^\calm_e < w_e - d(x,w)$. Thus, $w-d(x,w)\cdot{\bf 1}$ is an ultrametric
    coordinate-wise at most $x$ but not coordinate-wise at most $x^\calm$, contradicting
    that $x^\calm$ is the subdominant $\calm$-ultrametric.
    So (\ref{item:distance}) is proven and (\ref{item:canonicalclosest}) immediately follows.

    If (\ref{item:maximal}) were false and there existed some $\calm$-ultrametric $y \in C(x,\berg(\calm))$
    such that $y \ge x^\calm + \delta \cdot {\bf 1}$ with inequality somewhere,
    then $y - \delta \cdot {\bf 1}$ would not be coordinate-wise at most $x^\calm$.
    However, it would be coordinate-wise at most $x$,
    thus contradicting that $x^\calm$ is the subdominant $\calm$-ultrametric.
\end{proof}

\begin{defn}
    Given $x \in \rr^E$,
    we denote by $x^m$ the $l^\infty$-nearest ultrametric $x^\calm + d(x,\berg(\calm))\cdot {\bf 1}$
    and call it the \emph{maximal closest $\calm$-ultrametric to $x$}.
\end{defn}

\begin{ex}\label{ex:closest}
    Let $G$ be the graph displayed in Figure \ref{fig:maxlClosest} and denote its edge set by $E$.
    Let $x \in \rr^E$ be as on the left of Figure \ref{fig:maxlClosest}.
    Then the subdominant $\calm(G)$-ultrametric $x^\calm$
    and its translation giving the $l^\infty$-nearest $\calm(G)$-ultrametric $x^m$
    are shown to the right.
    \begin{figure}\centering
        \begin{subfigure}{0.32\textwidth}
            \begin{tikzpicture}
                \node at (-1,0){$x:$};
                \vertex (a) at (0,0)[label=left:$a$]{};
                \vertex (b) at (1,1)[label=above:$b$]{};
                \vertex (d) at (1,-1)[label=below:$d$]{};
                \vertex (c) at (2,0)[label=right:$c$]{};
                \path
                    (a) edge node[left,pos=.75]{$2$ \ } (b) 
                    (a) edge node[below]{$4$} (c)
                    (b) edge node[right,pos=0.25]{ \ $5$} (c)
                    (c) edge node[right,pos=0.75]{ \ $10$} (d)
                    (a) edge node[left,pos=0.75]{$6$ \ } (d)
                ;
            \end{tikzpicture}
        \end{subfigure}
        \begin{subfigure}{0.32\textwidth}
            \begin{tikzpicture}
                 \node at (-1,0){$x^\calm:$};
                \vertex (a) at (0,0)[label=left:$a$]{};
                \vertex (b) at (1,1)[label=above:$b$]{};
                \vertex (d) at (1,-1)[label=below:$d$]{};
                \vertex (c) at (2,0)[label=right:$c$]{};
                \path
                    (a) edge node[left,pos=.75]{$2$ \ } (b) 
                    (a) edge node[below]{$4$} (c)
                    (b) edge node[right,pos=0.25]{ \ $4$} (c)
                    (c) edge node[right,pos=0.75]{ \ $6$} (d)
                    (a) edge node[left,pos=0.75]{$6$ \ } (d)
                ;
            \end{tikzpicture}
        \end{subfigure}
        \begin{subfigure}{0.32\textwidth}
            \begin{tikzpicture}
                 \node at (-1,0){$x^m:$};
                \vertex (a) at (0,0)[label=left:$a$]{};
                \vertex (b) at (1,1)[label=above:$b$]{};
                \vertex (d) at (1,-1)[label=below:$d$]{};
                \vertex (c) at (2,0)[label=right:$c$]{};
                \path
                    (a) edge node[left,pos=.75]{$4$ \ } (b) 
                    (a) edge node[below]{$6$} (c)
                    (b) edge node[right,pos=0.25]{ \ $6$} (c)
                    (c) edge node[right,pos=0.75]{ \ $8$} (d)
                    (a) edge node[left,pos=0.75]{$8$ \ } (d)
                ;
            \end{tikzpicture}
        \end{subfigure}
        \caption{An element $x \in \rr^E$ alongside its subdominant $\calm(G)$-ultrametric $x^\calm$
        and the $l^\infty$-nearest $\calm$-ultrametric $x^m$.} \label{fig:maxlClosest}
    \end{figure}
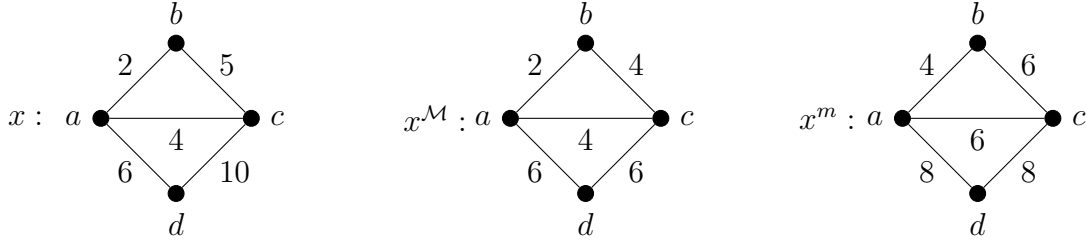
\end{ex}

Definition \ref{defn:mobileFlat} below introduces a way to decrease certain coordinates of an $\calm$-ultrametric $w$ that is $l^\infty$-nearest
to a given $x \in \rr^E$ to produce another $\calm$-ultrametric $l^\infty$-nearest to $x$.
The coordinates of $w$ that can be decreased are determined by what we will call \emph{mobile flats}.
We call the process of decreasing these coordinates \emph{sliding mobile flats (all the way) down}.
Theorem \ref{thm:verticesAlgorithm} uses these concepts
to describe a generating set of $C(x,\berg(\calm))$.

\begin{defn}\label{defn:mobileFlat}
    Let $\calm$ be a matroid on ground set $E$.
    Let $x \in \rr^E$ and let $w \in \berg(\calm)$ be $l^\infty$-nearest to $x$.
    Let $\cals$ be a resolution of $\calt(w)$ and $\alpha: \cals \rightarrow \rr$
    be compatible with $\cals$ satisfying $w = w^{\cals,\alpha}$.
    We say that $F \in \cals$ is \emph{mobile} if there exists an $\calm$-ultrametric $\hat w \neq w$
    expressible as $\hat w = w^{\cals,\hat \alpha}$ with $\hat \alpha$ compatible with $\cals$ such that
    \begin{enumerate}
        \item $\hat w$ is also nearest to $x$ in the $l^\infty$-norm
        \item $\hat \alpha(G) = \alpha(G)$ for all $G \neq F$, and
        \item $\hat \alpha(F) < \alpha(F)$.
    \end{enumerate}
    In this case, we say that $\hat w$ is obtained from $w$ by \emph{sliding $F$ down}.
    If moreover $F$ is no longer mobile in $w^{\cals,\hat \alpha}$, i.e. if
    $\hat\alpha(F) = \max\{\alpha(G): G \in \cals \textnormal{ and } G \subsetneq F\}$
    or $\hat \alpha(F)$ is the minimum value such that $w^{\cals,\hat \alpha}$ is $l^\infty$-nearest to $x$,
    then we say that $\hat w$ is obtained from $w$ by \emph{sliding $F$ all the way down}.
\end{defn}

\begin{rmk}\label{rmk:isMobile}
	Given some $x \in \rr^E$ and some $w^{\cals,\alpha}$ that is $l^\infty$-nearest to $x$,
	one can determine that a given $F \in \cals$ is mobile by decreasing $\alpha(F)$ by some small
	$\varepsilon > 0$ and seeing that the resulting $\calm$-ultrametric is still $l^\infty$-nearest to $x$.
\end{rmk}

\begin{rmk}\label{rmk:mobileInPolytomy}
	If $\cals$ is a resolution of $\calt(w)$ and $F \in \cals \setminus \calt(w)$
	is mobile, then $F$ is contained in a polytomy of $\calt(w)$
	and all elements of $\cals$ covered by $F$ are also in $\calt(w)$.
\end{rmk}

\begin{thm}\label{thm:verticesAlgorithm}
    Let $\calm$ be a matroid on ground set $E$ and let $x \in \rr^E$.
    Define $S_0 := \{x^m\}$ and for each $i \ge 1$,
    define $S_i$ to be the set of $\calm$-ultrametrics obtained from some $w \in S_{i-1}$
    by sliding a mobile flat in a resolution of $\calt(w)$ all the way down.
    Then
    \begin{enumerate}
        \item $\bigcup_i S_i$ is a finite set,
        \item the tropical convex hull of $\bigcup_i S_i$ is $C(x,\berg(\calm))$, and
        \item each tropical vertex $v$ of $C(x,\berg(\calm))$ has at most one mobile flat across
        all resolutions of $\calt(v)$.
    \end{enumerate}
\end{thm}
\begin{proof}
    We first prove that $\bigcup_i S_i$ is a finite set.
    Let $w \in S_i$ for some $i \ge 0$.
    Then each coordinate $w_e$ is either $x^m_f$
    or $x^m_f - d(x,\berg(\calm))$ for some $f \in E$, not necessarily equal to $e$.
    So as $w$ ranges over $\bigcup_i S_i$, there are only finitely many values that each $w_e$ can take
    and so $\bigcup_i S_i$ is a finite set.

    We now prove that each tropical vertex has at most one mobile flat.
    Let $v \in C(x,\berg(\calm))$.
    Let $\alpha$ be such that $v = w^{\calt(v),\alpha}$ (recall Definition \ref{defn:multrametricFromTopology}).
    If $\cals_1$ and $\cals_2$ are resolutions of $\calt(v)$
    and $F_i \in \cals_i$ is mobile,
    then there exist $\alpha_i: \cals_i \rightarrow \rr$
    compatible with $\cals_i$ such that $w^{\cals_i,\alpha_i} \in C(x,\berg(\calm))$,
    and $w^{\cals_i,\alpha_i}_e = v_e-\varepsilon$ for a fixed small $\varepsilon > 0$
    whenever $e \in F_i \setminus \bigcup_F F$
    where the union is taken over all $F \in S_i$ such that $F \subsetneq F_i$,
    and $w^{\cals_i,\alpha_i}_e = v_e$ for all other $e \in E$.
    We claim that $w^{\cals_1,\alpha_1}_e \neq v_e$ implies $w^{\cals_2,\alpha_2}_e = v_e$.
    When $F_1$ and $F_2$ are disjoint, the claim is obvious.
    When $F_1$ and $F_2$ are not disjoint,
    they must be subsets of the same polytomy $F \in \calt(v)$.
    Let $U$ be the union of all the flats covered by $F$ in $\calt(v)$.
    Then $U \subseteq F_1 \cap F_2$.
    Moreover, $U = F_1 \cap F_2$ because if
    $e \in F_1\cap F_2 \setminus U$,
    then $\rank(F_1 /U) > \rank((F_1 \cap F_2)/U) \ge \rank((U \cup \{e\})/U) = 1$).
    In light of Remark \ref{rmk:mobileInPolytomy},
    this is a contradiction because then $F_1$ would be a polytomy in $\cals_1$.
    The claim then follows because $w^{\cals_i,\alpha_i}_e \neq v_e$
    if and only if $e \in F_i\setminus U$.
    Now we have $v = w^{\cals_1,\alpha_1} \oplus w^{\cals_2,\alpha_2}$
    and so $v$ is not a tropical vertex of $C(x,\berg(\calm))$.

    Now we prove that the tropical convex hull of $\bigcup_i S_i$ is $C(x,\berg(\calm))$
    by showing that each vertex of $C(x,\berg(\calm))$ is a member of some $S_i$.
    So let $v$ be a tropical vertex of $C(x,\berg(\calm))$.
    We construct a sequence $x^m = w^0 \ge w^1 \ge \dots \ge v$
    such that $w^i \in S_i$ and $w^i \neq w^{i+1}$.
    Since $\bigcup_i S_i$ is finite, this sequence must eventually terminate
    and so the final $w^i$ is equal to $v$.
    Assuming $w^i$ has been constructed and satisfies $w^i \ge v$ and $w^i \neq v$,
    we show how to construct $w^{i+1}$ satisfying $w^i \ge w^{i+1} \ge v$ and $w^{i+1} \neq w^i$.

    First assume $\calt(w^i) \subseteq \calt(v)$.
    Let $\cals$ be a resolution of $\calt(v)$.
    Then $\cals$ is also a resolution of $\calt(w_i)$.
    Let $\alpha_{w^i}, \alpha_v$ be such that
    $w^{i} = w^{\cals,\alpha_{w^i}}$ and $v = w^{\cals,\alpha_v}$.
    Let $F \in \cals$ be a minimal element such that $\alpha_v(F) < \alpha_{w^i}(F)$.
    We can choose such an $F$ to be non-mobile in $v$.
    Otherwise, the unique mobile flat in $\cals$ of $w^i$ would be $F$,
    which would also be the unique mobile flat of $\cals$ in $v$
    and so for all $G \in \cals \setminus \{F\}$,
    $\alpha_{w^i}(G) = \alpha_v(G)$.
    Since $F$ is mobile in $v$,
    there exists some $\alpha: \cals \rightarrow \rr$ compatible with $\cals$
    such that $\alpha(G) = \alpha_v(G)$ for $G \neq F$
    but $\alpha(F) < \alpha_v(F)$ and $w^{\cals,\alpha} \in C(x,\berg(\calm))$.
    This contradicts $v$ being a vertex of $C(x,\berg(\calm))$
    because $v = (\alpha_v(F) - \alpha_{w^i}(F))\odot w^i \oplus w^{\cals,\alpha}$.
    So we can choose $F$ to be mobile in $w^{i}$ and not in $v$.
    Define $\alpha_{w^{i+1}}: \cals \rightarrow \rr$
    by $\alpha_{w^{i+1}}(G) = \alpha_{w^i}(G)$ when $G \neq F$
    and $\alpha_{w^{i+1}}(F) = \alpha_v(F)$.
    Define $w^{i+1} := w^{\cals,\alpha_{w^{i+1}}}$.
    Then $w^i \ge w^{i+1} \ge v$ and $w^{i+1}$
    is obtained from $w^i$ by sliding $F$ down.
    Since $F$ was chosen to be minimal such that $\alpha_v(F) < \alpha_{w^i}(F)$
    and $\alpha_{w^{i+1}}(G) = \alpha_{w^i}(G)$ when $G \neq F$,
    non-mobility of $F$ in $v$ implies non-mobility of $F$ in $w^{i+1}$.
    Hence $w^{i+1}$ is obtained from $w^{i}$ by sliding $F$ \emph{all the way} down
    and so $w^{i+1} \in S_{i+1}$.

    Now assume $\calt(w^i) \nsubseteq \calt(v)$.
    Denote $v^t := (t \odot w^i) \oplus v$.
    Since $C(x,\berg(\calm))$ is tropically convex,
    $v^t \in C(x,\berg(\calm))$ whenever $t < 0$.
    Let $t_0 < 0 $ maximum such that $\calt(w^i)\setminus \calt(v^{t_0})$ is nonempty
    and let $G \in \calt(w^i)\setminus \calt(v^{t_0})$ be maximal.
    Note that for small $\varepsilon > 0$,
    $\calt(w^i) \subseteq \calt(v^{t_0+\varepsilon})$
    and the minimal $H \in \calt(v^{t_0+\varepsilon})$ that strictly contains $G$
    is also a member of $\calt(v^{t_0})$.
    Let $\cals$ be a resolution of $\calt(v^{t_0+\varepsilon})$
    and therefore also a resolution of $\calt(w^i)$.
    Choose $K \in \cals$ such that $G \subsetneq K \subseteq H$
    and let $w^{i+1}$ be the result of sliding $K$ all the way down in $w^i$.
    Then $w^{i+1} \in S_{i+1}$ and $w^i \ge w^{i+1} \ge v^t \ge v$.
\end{proof}

{ As with Theorem \ref{thm:treesAlgorithm}, the set of $\calm$-ultrametrics
specified by Theorem~\ref{thm:verticesAlgorithm}(3) is, in general, a strict superset of
the set of tropical vertices; see \cite{yu2019extreme}.}

\begin{ex}\label{ex:multrametricsExample}
    Let $G$ be the graph from Figure \ref{fig:maxlClosest} and let $x$ be the edge-weighting displayed.
    We now describe how to use Theorem \ref{thm:verticesAlgorithm} to obtain a
    generating set of the tropical polytope consisting of the $\calm(G)$-ultrametrics
    that are $l^\infty$-nearest to $x$.
    Figure \ref{fig:theoremExample} shows the $\calm(G)$-ultrametrics in each nonempty $S_i$, displayed on their topologies.
    The mobile flats of the unique element $x^m$ of $S_0$ are $\{ab,ac,bc\}$ and $\{ab\}$.
    Sliding $\{ab\}$ all the way down yields the element of $S_1$ shown on the left,
    and sliding $\{ab,ac,bc\}$ all the way down yields the element of $S_1$ shown on the right.
    The only mobile flat of the element of $S_1$ shown on the left is $\{ab,ac,bc\}$.
    Sliding this all the way down yields the left-most element displayed in $S_2$.
    The element of $S_1$ shown on the right has $\{ab,ac,bc\}$ as a polytomy.
    There are three possible resolutions, the first obtained by adding the flat $\{ab\}$,
    the second by adding $\{ac\}$ and the third by adding $\{bc\}$.
    Each such flat is mobile, and the elements of $S_2$ obtained by sliding each all the way down
    are shown second, third, and fourth from the left in $S_2$.
    Continuing in this way yields the elements shown in $S_3$ and $S_4$.
    Note that there are no mobile flats in any element of $S_4$ so $S_i$ is empty for $i \ge 5$.
    The leftmost element of $S_2$ also appears in $S_3$ and $S_4$.
    A subset of $\bigcup_i S_i$ whose tropical convex hull is $C(x,\berg(\calm(G)))$ is shown in red.
    Note that we've omitted elements with two or more mobile flats, as well as repeated elements.
    \begin{figure}[h]\centering
        \begin{subfigure}{0.08\textwidth}\centering
            $S_0$
        \end{subfigure}
        \begin{subfigure}{0.9\textwidth}\centering
            \begin{tikzpicture}
                \node (a) at (0,2){$8 \ \{ab,ac,ad,bd,cd\}$};
                \node (b) at (0,1){$6 \ \{ab,ac,bc\}$};
                \node (c) at (0,0){$4 \ \{ab\}$};
                \draw (a) -- (b);
                \draw (b) -- (c);
            \end{tikzpicture}
        \end{subfigure}
        \\ \ \\ \ \\
        \begin{subfigure}{0.08\textwidth}\centering
            $S_1$
        \end{subfigure}
        \begin{subfigure}{0.45\textwidth}\centering
            \begin{tikzpicture}\color{red}
                \node (a) at (0,2){$8 \ \{ab,ac,ad,bd,cd\}$};
                \node (b) at (0,1){$6 \ \{ab,ac,bc\}$};
                \node (c) at (0,0){$0 \ \{ab\}$};
                \draw (a) -- (b);
                \draw (b) -- (c);
            \end{tikzpicture}
        \end{subfigure}
        \begin{subfigure}{0.45\textwidth}\centering
            \begin{tikzpicture}
                \node (a) at (0,2){$8 \ \{ab,ac,ad,bd,cd\}$};
                \node (b) at (0,1){$4 \ \{ab,ac,bc\}$};
                \draw (a) -- (b);
            \end{tikzpicture}
        \end{subfigure}
        \\ \ \\ \ \\
        \begin{subfigure}{0.08\textwidth}\centering
                $S_2$
        \end{subfigure}
        \begin{subfigure}{0.22\textwidth}\centering
            \begin{tikzpicture}\color{red}
                \node (a) at (0,2){$8 \ E \ \ $};
                \node (b) at (0,1){$3 \ \{ab,ac,bc\}$};
                \node (c) at (0,0){$0 \ \{ab\}$};
                \draw (a) -- (b);
                \draw (b) -- (c);
            \end{tikzpicture}
        \end{subfigure}
        \begin{subfigure}{0.22\textwidth}\centering
            \begin{tikzpicture}\color{red}
                \node (a) at (0,2){$8 \ E \ \ $};
                \node (b) at (0,1){$4 \ \{ab,ac,bc\}$};
                \node (c) at (0,0){$0 \ \{ab\}$};
                \draw (a) -- (b);
                \draw (b) -- (c);
            \end{tikzpicture}
        \end{subfigure}
        \begin{subfigure}{0.22\textwidth}\centering
            \begin{tikzpicture}\color{red}
                \node (a) at (0,2){$8 \ E \ \ $};
                \node (b) at (0,1){$4 \ \{ab,ac,bc\}$};
                \node (c) at (0,0){$2 \ \{ac\}$};
                \draw (a) -- (b);
                \draw (b) -- (c);
            \end{tikzpicture}
        \end{subfigure}
        \begin{subfigure}{0.22\textwidth}\centering
            \begin{tikzpicture}\color{red}
                \node (a) at (0,2){$8 \ E \ \ $};
                \node (b) at (0,1){$4 \ \{ab,ac,bc\}$};
                \node (c) at (0,0){$3 \ \{bc\}$};
                \draw (a) -- (b);
                \draw (b) -- (c);
            \end{tikzpicture}
        \end{subfigure}
        \\ \ \\ \ \\
        \begin{subfigure}{0.08\textwidth}\centering
                $S_3$ \ \ 
        \end{subfigure}
        \begin{subfigure}{0.29\textwidth}\centering
            \begin{tikzpicture}
                \node (a) at (0,2){$8 \ \{ab,ac,ad,bd,cd\}$};
                \node (b) at (0,1){$3 \ \{ab,ac,bc\}$};
                \node (c) at (0,0){$0 \ \{ab\}$};
                \draw (a) -- (b);
                \draw (b) -- (c);
            \end{tikzpicture}
        \end{subfigure}
        \begin{subfigure}{0.29\textwidth}\centering
            \begin{tikzpicture}\color{red}
                \node (a) at (0,2){$8 \ \{ab,ac,ad,bd,cd\}$};
                \node (b) at (0,1){$3 \ \{ab,ac,bc\}$};
                \node (c) at (0,0){$2 \ \{ac\}$};
                \draw (a) -- (b);
                \draw (b) -- (c);
            \end{tikzpicture}
        \end{subfigure}
        \begin{subfigure}{0.29\textwidth}\centering
            \begin{tikzpicture}
                \node (a) at (0,2){$8 \ \{ab,ac,ad,bd,cd\}$};
                \node (b) at (0,1){$3 \ \{ab,ac,bc\}$};
                \draw (a) -- (b);
            \end{tikzpicture}
        \end{subfigure}
        \\ \ \\ \ \\
        \begin{subfigure}{0.08\textwidth}\centering
                $S_4$
        \end{subfigure}
        \begin{subfigure}{0.45\textwidth}\centering
            \begin{tikzpicture}\color{red}
                \node (a) at (0,2){$8 \ \{ab,ac,ad,bd,cd\}$};
                \node (b) at (0,1){$3 \ \{ab,ac,bc\}$};
                \node (c) at (0,0){$2 \ \{ac\}$};
                \draw (a) -- (b);
                \draw (b) -- (c);
            \end{tikzpicture}
        \end{subfigure}
        \begin{subfigure}{0.45\textwidth}\centering
            \begin{tikzpicture}
                \node (a) at (0,2){$8 \ \{ab,ac,ad,bd,cd\}$};
                \node (b) at (0,1){$3 \ \{ab,ac,bc\}$};
                \node (c) at (0,0){$0 \ \{ab\}$};
                \draw (a) -- (b);
                \draw (b) -- (c);
            \end{tikzpicture}
        \end{subfigure}
        \caption{The nonempty $S_i$'s from Theorem \ref{thm:verticesAlgorithm} for the edge-weighting of the
        graph $G$ in Figure \ref{fig:maxlClosest}.}\label{fig:theoremExample}
    \end{figure}
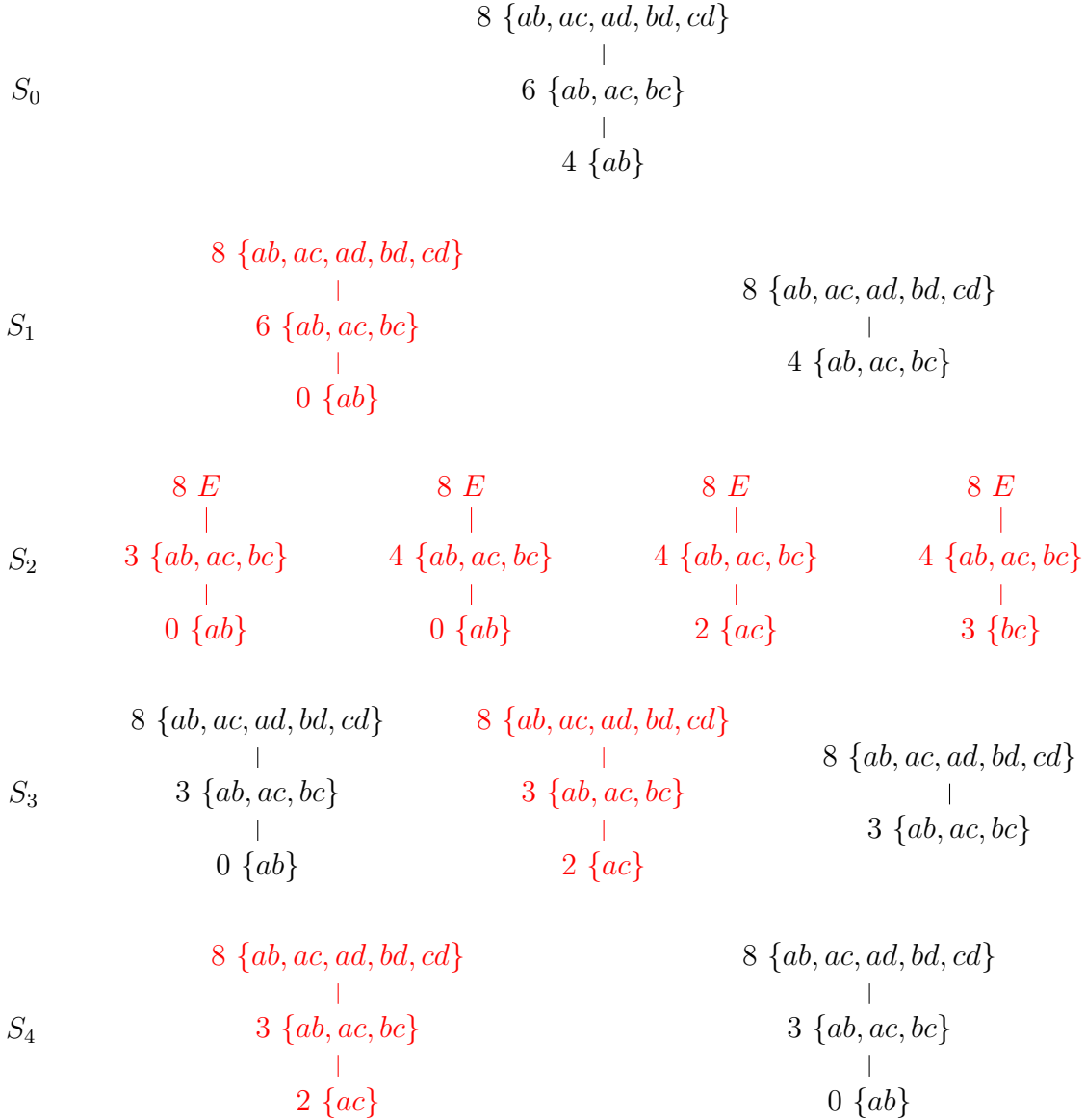
\end{ex}

The following proposition tells us
that if all the $\calm$-ultrametrics in the generating set of $C(x,\berg(\calm))$ indicated
by Theorem \ref{thm:verticesAlgorithm} have the same topology,
then all elements of $C(x,\berg(\calm))$ have the same topology.

\begin{prop}\label{prop:multipletopologies}
    Let $\calm$ be a matroid on ground set $E$ and let $x \in \rr^E$.
    Then set of all $\calm$-ultrametrics that are $l^\infty$-nearest to $x$
    have the same topology if and only if all tropical vertices of $C(x,\berg(\calm))$
    have the same topology.
\end{prop}
\begin{proof}
	This follows immediately from Proposition \ref{prop:topologyconvex}.
\end{proof}

When $\calm := \calm(G)$ is the matroid underlying some graph $G$,
then Theorem \ref{thm:verticesAlgorithm} has potential use for phylogenetics even when $G$
is not the complete graph.
In particular, it sometimes happens that only a subset of the pairwise distances
between $n$ species can be computed within a reasonable budget.
Then one may ask the question of which partial ultrametrics are $l^\infty$-nearest to
the observed distances.
Assuming that the observed distances correspond to the edge set $E$ of a graph $G$,
the following proposition tells us that the above question
is equivalent to: given some partial dissimilarity map $x \in \rr^E$,
which $\calm(G)$-ultrametrics are $l^\infty$-nearest to $x$?

\begin{prop}\label{prop:completion}
    Let $E \subsetneq \binom{[n]}{2}$,
    let $G$ be the graph with vertex set $[n]$ and edge set $E$,
    and let $x \in \rr^E$.
    Then we may extend $x$ to some ultrametric $x' \in \berg(\calm(K_n))$
    if and only if $x$ is an $\calm(G)$-ultrametric.
\end{prop}
\begin{proof}
    First let $x$ be a $\calm(G)$-ultrametric.
    Let $e \in \binom{[n]}{2} \setminus E$.
    Let $G'$ be the graph obtained by adding $e$ to $G$.
    We can extend $x$ to a $\calm(G')$-ultrametric $x'$
    by setting $x_e'$ to be the maximum
    of all the minimum edge weights appearing in some cocircuit of $\calm(G')$.
    That this is indeed an $\calm(G')$-ultrametric follows
    from Ardila's characterization of $\calm$-ultrametrics in terms of $\calm$'s cocircuits \cite{ardila2004}.
    By induction it follows that $x$ may be completed to an $\calm(K_n)$-ultrametric.
    \\
    \indent
    Now let $x \in \rr^E$ and assume that there exists some $x' \in \berg(\calm(K_n))$
    such that $x_e = x'_e$ for each $e \in E$.
    Since $x'$ is an $\calm(K_n)$-ultrametric, each $e \in E$ appears in some $x'$-minimal
    basis of $\calm(K_n)$.
    As $x'_e = x_e$ for each $e \in E$,
    it follows that each $e \in E$ appears in some $x$-minimal
    basis of $\calm(G)$.
    Therefore $x$ is an $\calm(G)$-ultrametric.
\end{proof}


\section{Example on a biological dataset}\label{sec:data}
    { Now that we understand how uniqueness of the $l^\infty$-nearest ultrametric
    can fail to be unique, one might wonder if this is likely to happen for a dissimilarity map
    not explicitly constructed to break uniqueness.
    To this end,} we now apply Theorem \ref{thm:treesAlgorithm} to the dataset 
    displayed in Figure \ref{fig:data}.
    It consists of pairwise immunological distances between the species
    dog, bear, raccoon, weasel, seal, sea lion, cat, and monkey
    that were obtained by Sarich in \cite{sarich1969pinniped}.
    It is used in the textbook \cite{felsenstein2004inferring} to illustrate the
    UPGMA and neighbor joining algorithms,
    which are two other distance-based methods for phylogenetic reconstruction.

    \begin{figure}
    \begin{tabular}{c| c c c c c c c c}
        & dog & bear & raccoon & weasel & seal & sea lion & cat & monkey\\
        \hline
        dog&0&32&48&51&50&48&98&148\\
        bear&32&0&26&34&29&33&84&136\\
        raccoon&48&26&0&42&44&44&92&152\\
        weasel&51&34&42&0&44&38&86&142\\
        seal&50&29&44&44&0&24&89&142\\
        sea lion&48&33&44&38&24&0&90&142\\
        cat&98&84&92&86&89&90&0&148\\
        monkey&148&136&152&142&142&142&148&0
    \end{tabular}
    \caption{Pairwise immunological distances between eight species.}\label{fig:data}
    \end{figure}

    \begin{figure}
        \begin{tikzpicture}[scale=0.42]
            \node (8) at (0,0)[label=below:M]{};
            \node (7) at (2,0)[label=below:C]{};
            \node (4) at (4,0)[label=below:D]{};
            \node (1) at (6,0)[label=below:W]{};
            \node (2) at (8,0)[label=below:B]{};
            \node (3) at (10,0)[label=below:R]{};
            \node (5) at (12,0)[label=below:S]{};
            \node (6) at (14,0)[label=below:SL]{};
            \node (23) at (9,1)[label=above:26]{};
            \node (56) at (13,1)[label=above:24]{};
            \node (2356) at (11.5,2)[label=above:37.5]{};
            \node (12356) at (9,3)[label=above:39.5]{};
            \node (123456) at (7,4)[label=above:45.8]{};
            \node (1234567) at (5,5)[label=above:89.8]{};
            \node (12345678) at (3,6)[label=above:144.3]{};
            \draw (2.center) -- (23.center) -- (3.center);
            \draw (5.center) -- (56.center) -- (6.center);
            \draw (23.center) -- (2356.center) -- (56.center);
            \draw (1.center) -- (12356.center) -- (2356.center);
            \draw (4.center) -- (123456.center) -- (12356.center);
            \draw (7.center) -- (1234567.center) -- (123456.center);
            \draw (8.center) -- (12345678.center) -- (1234567.center);
        \end{tikzpicture}
        \caption{Ultrametric returned by the UPGMA algorithm.}\label{fig:upgma}
    \end{figure}
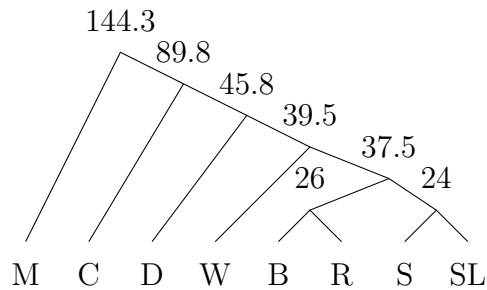
    
    Theorem \ref{thm:verticesAlgorithm} suggests an algorithm for computing
    a generating set of the set of ultrametrics $l^\infty$-nearest to a given dissimilarity map.
    This consists of computing all nonempty $S_i$'s and removing all ultrametrics that have more than one mobile internal node.
    Applying this to the dataset in Figure \ref{fig:data}
    gives us the twenty ultrametrics displayed in Table \ref{table:tropicalVertices}.
    Four different tree topologies appear;
    {for example, note that the topologies of the first, second, eighth, and fourteenth
    ultrametrics in the row-major order of Table~\ref{table:tropicalVertices} are distinct.}

    The UPGMA algorithm always returns an ultrametric.
    Figure \ref{fig:upgma} shows the ultrametric computed by the UPGMA algorithm when applied to the dataset
    given in Figure \ref{fig:data} (see \cite[pp.162-166]{felsenstein2004inferring}).
    No ultrametric sharing the topology of the ultrametric shown in Figure \ref{fig:upgma}
    will be $l^\infty$-nearest to the data.
    To see this, note that among the ultrametrics displayed in Table \ref{table:tropicalVertices},
    the distance between weasel and seal is 42 or 43,
    and that the distance between dog and seal is always 41.
    Since the set of $l^\infty$-nearest ultrametrics is tropically convex,
    any ultrametric $l^\infty$-nearest to the data will have the distance between
    weasel and seal strictly greater than the distance between dog and seal.
    However, the opposite relation will be true in any ultrametric whose topology
    is the tree displayed in Figure \ref{fig:upgma}.

    \newpage
	\begin{center}
    {\small
    \begin{longtable}{l l l}
        \caption[Superset of tropical vertices]{A set of ultrametrics
        whose tropical convex hull is the set of ultrametrics $l^\infty$-nearest to the dataset in Figure \ref{fig:data}.} \label{table:tropicalVertices} \\ 
        \endfirsthead
        \multicolumn{3}{c}%
        {{\tablename\ \thetable{} -- continued from previous page}}\\
        \hline
        \endhead
        \hline \multicolumn{3}{r}{{Continued on next page}}
        \endfoot
        \endlastfoot
        \begin{tikzpicture}[scale=0.42]
            \node (8) at (0,0)[label=below:M]{};
            \node (7) at (2,0)[label=below:C]{};
            \node (4) at (4,0)[label=below:W]{};
            \node (1) at (6,0)[label=below:D]{};
            \node (2) at (8,0)[label=below:B]{};
            \node (3) at (10,0)[label=below:R]{};
            \node (5) at (12,0)[label=below:S]{};
            \node (6) at (14,0)[label=below:SL]{};
            \node (23) at (9,1)[label=above:17]{};
            \node (56) at (13,1)[label=above:15]{};
            \node (2356) at (11.5,2)[label=above:35]{};
            \node (12356) at (9,3)[label=above:41]{};
            \node (123456) at (7,4)[label=above:42]{};
            \node (1234567) at (5,5)[label=above:89]{};
            \node (12345678) at (3,6)[label=above:143]{};
            \draw (2.center) -- (23.center) -- (3.center);
            \draw (5.center) -- (56.center) -- (6.center);
            \draw (23.center) -- (2356.center) -- (56.center);
            \draw (1.center) -- (12356.center) -- (2356.center);
            \draw (4.center) -- (123456.center) -- (12356.center);
            \draw (7.center) -- (1234567.center) -- (123456.center);
            \draw (8.center) -- (12345678.center) -- (1234567.center);
        \end{tikzpicture}& &
        \begin{tikzpicture}[scale=0.42]
            \node (8) at (0,0)[label=below:M]{};
            \node (7) at (2,0)[label=below:C]{};
            \node (4) at (4,0)[label=below:W]{};
            \node (1) at (6,0)[label=below:D]{};
            \node (2) at (8,0)[label=below:B]{};
            \node (3) at (10,0)[label=below:R]{};
            \node (5) at (12,0)[label=below:S]{};
            \node (6) at (14,0)[label=below:SL]{};
            \node (56) at (13,1)[label=above:15]{};
            \node (2356) at (11.5,2)[label=above:35]{};
            \node (12356) at (9,3)[label=above:41]{};
            \node (123456) at (7,4)[label=above:42]{};
            \node (1234567) at (5,5)[label=above:89]{};
            \node (12345678) at (3,6)[label=above:143]{};
            \draw (2.center) -- (2356.center) -- (3.center);
            \draw (5.center) -- (56.center) -- (6.center);
            \draw (2356.center) -- (56.center);
            \draw (1.center) -- (12356.center) -- (2356.center);
            \draw (4.center) -- (123456.center) -- (12356.center);
            \draw (7.center) -- (1234567.center) -- (123456.center);
            \draw (8.center) -- (12345678.center) -- (1234567.center);
        \end{tikzpicture}\\
        \begin{tikzpicture}[scale=0.42]
            \node (8) at (0,0)[label=below:M]{};
            \node (7) at (2,0)[label=below:C]{};
            \node (4) at (4,0)[label=below:W]{};
            \node (1) at (6,0)[label=below:D]{};
            \node (2) at (8,0)[label=below:B]{};
            \node (3) at (10,0)[label=below:R]{};
            \node (5) at (12,0)[label=below:S]{};
            \node (6) at (14,0)[label=below:SL]{};
            \node (23) at (9,1)[label=above:17]{};
            \node (56) at (13,1)[label=above:33]{};
            \node (2356) at (11.5,2)[label=above:35]{};
            \node (12356) at (9,3)[label=above:41]{};
            \node (123456) at (7,4)[label=above:42]{};
            \node (1234567) at (5,5)[label=above:89]{};
            \node (12345678) at (3,6)[label=above:143]{};
            \draw (2.center) -- (23.center) -- (3.center);
            \draw (5.center) -- (56.center) -- (6.center);
            \draw (23.center) -- (2356.center) -- (56.center);
            \draw (1.center) -- (12356.center) -- (2356.center);
            \draw (4.center) -- (123456.center) -- (12356.center);
            \draw (7.center) -- (1234567.center) -- (123456.center);
            \draw (8.center) -- (12345678.center) -- (1234567.center);
        \end{tikzpicture}& &
        \begin{tikzpicture}[scale=0.42]
            \node (8) at (0,0)[label=below:M]{};
            \node (7) at (2,0)[label=below:C]{};
            \node (4) at (4,0)[label=below:W]{};
            \node (1) at (6,0)[label=below:D]{};
            \node (2) at (8,0)[label=below:B]{};
            \node (3) at (10,0)[label=below:R]{};
            \node (5) at (12,0)[label=below:S]{};
            \node (6) at (14,0)[label=below:SL]{};
            \node (23) at (9,1)[label=above:17]{};
            \node (56) at (13,1)[label=above:15]{};
            \node (2356) at (11.5,2)[label=above:38]{};
            \node (12356) at (9,3)[label=above:41]{};
            \node (123456) at (7,4)[label=above:42]{};
            \node (1234567) at (5,5)[label=above:89]{};
            \node (12345678) at (3,6)[label=above:143]{};
            \draw (2.center) -- (23.center) -- (3.center);
            \draw (5.center) -- (56.center) -- (6.center);
            \draw (23.center) -- (2356.center) -- (56.center);
            \draw (1.center) -- (12356.center) -- (2356.center);
            \draw (4.center) -- (123456.center) -- (12356.center);
            \draw (7.center) -- (1234567.center) -- (123456.center);
            \draw (8.center) -- (12345678.center) -- (1234567.center);
        \end{tikzpicture} \\
        \begin{tikzpicture}[scale=0.42]
            \node (8) at (0,0)[label=below:M]{};
            \node (7) at (2,0)[label=below:C]{};
            \node (4) at (4,0)[label=below:W]{};
            \node (1) at (6,0)[label=below:D]{};
            \node (2) at (8,0)[label=below:B]{};
            \node (3) at (10,0)[label=below:R]{};
            \node (5) at (12,0)[label=below:S]{};
            \node (6) at (14,0)[label=below:SL]{};
            \node (23) at (9,1)[label=above:17]{};
            \node (56) at (13,1)[label=above:15]{};
            \node (2356) at (11.5,2)[label=above:35]{};
            \node (12356) at (9,3)[label=above:41]{};
            \node (123456) at (7,4)[label=above:43]{};
            \node (1234567) at (5,5)[label=above:89]{};
            \node (12345678) at (3,6)[label=above:143]{};
            \draw (2.center) -- (23.center) -- (3.center);
            \draw (5.center) -- (56.center) -- (6.center);
            \draw (23.center) -- (2356.center) -- (56.center);
            \draw (1.center) -- (12356.center) -- (2356.center);
            \draw (4.center) -- (123456.center) -- (12356.center);
            \draw (7.center) -- (1234567.center) -- (123456.center);
            \draw (8.center) -- (12345678.center) -- (1234567.center);
        \end{tikzpicture}& & 
        \begin{tikzpicture}[scale=0.42]
            \node (8) at (0,0)[label=below:M]{};
            \node (7) at (2,0)[label=below:C]{};
            \node (4) at (4,0)[label=below:W]{};
            \node (1) at (6,0)[label=below:D]{};
            \node (2) at (8,0)[label=below:B]{};
            \node (3) at (10,0)[label=below:R]{};
            \node (5) at (12,0)[label=below:S]{};
            \node (6) at (14,0)[label=below:SL]{};
            \node (23) at (9,1)[label=above:17]{};
            \node (56) at (13,1)[label=above:15]{};
            \node (2356) at (11.5,2)[label=above:35]{};
            \node (12356) at (9,3)[label=above:41]{};
            \node (123456) at (7,4)[label=above:42]{};
            \node (1234567) at (5,5)[label=above:93]{};
            \node (12345678) at (3,6)[label=above:143]{};
            \draw (2.center) -- (23.center) -- (3.center);
            \draw (5.center) -- (56.center) -- (6.center);
            \draw (23.center) -- (2356.center) -- (56.center);
            \draw (1.center) -- (12356.center) -- (2356.center);
            \draw (4.center) -- (123456.center) -- (12356.center);
            \draw (7.center) -- (1234567.center) -- (123456.center);
            \draw (8.center) -- (12345678.center) -- (1234567.center);
        \end{tikzpicture} \\
        \begin{tikzpicture}[scale=0.42]
            \node (8) at (0,0)[label=below:M]{};
            \node (7) at (2,0)[label=below:C]{};
            \node (4) at (4,0)[label=below:W]{};
            \node (1) at (6,0)[label=below:D]{};
            \node (2) at (8,0)[label=below:B]{};
            \node (3) at (10,0)[label=below:R]{};
            \node (5) at (12,0)[label=below:S]{};
            \node (6) at (14,0)[label=below:SL]{};
            \node (23) at (9,1)[label=above:17]{};
            \node (56) at (13,1)[label=above:15]{};
            \node (2356) at (11.5,2)[label=above:35]{};
            \node (12356) at (9,3)[label=above:41]{};
            \node (123456) at (7,4)[label=above:42]{};
            \node (1234567) at (5,5)[label=above:93]{};
            \node (12345678) at (3,6)[label=above:145]{};
            \draw (2.center) -- (23.center) -- (3.center);
            \draw (5.center) -- (56.center) -- (6.center);
            \draw (23.center) -- (2356.center) -- (56.center);
            \draw (1.center) -- (12356.center) -- (2356.center);
            \draw (4.center) -- (123456.center) -- (12356.center);
            \draw (7.center) -- (1234567.center) -- (123456.center);
            \draw (8.center) -- (12345678.center) -- (1234567.center);
        \end{tikzpicture} & &
        \begin{tikzpicture}[scale=0.42]
            \node (8) at (0,0)[label=below:M]{};
            \node (7) at (2,0)[label=below:C]{};
            \node (4) at (4,0)[label=below:W]{};
            \node (1) at (6,0)[label=below:D]{};
            \node (2) at (10,0)[label=below:B]{};
            \node (3) at (8,0)[label=below:R]{};
            \node (5) at (12,0)[label=below:S]{};
            \node (6) at (14,0)[label=below:SL]{};
            \node (56) at (3+60/6,15/2-1/2*(3+60/6)[label=above:15]{};
            \node (256) at (3+50/6,15/2-1/2*(3+50/6)[label=above:24]{};
            \node (2356) at (3+40/6,15/2-1/2*(3+40/6)[label=above:35]{};
            \node (12356) at (3+30/6,15/2-1/2*(3+30/6)[label=above:41]{};
            \node (123456) at (3+20/6,15/2-1/2*(3+20/6)[label=above:42]{};
            \node (1234567) at (3 + 10/6,15/2-1/2*(3+10/6)[label=above:89]{};
            \node (12345678) at (3,6)[label=above:143]{};
            \draw (5.center) -- (56.center) -- (6.center);
            \draw (2.center) -- (256.center) -- (56.center);
            \draw (3.center) -- (2356.center) -- (256.center);
            \draw (1.center) -- (12356.center) -- (2356.center);
            \draw (4.center) -- (123456.center) -- (12356.center);
            \draw (7.center) -- (1234567.center) -- (123456.center);
            \draw (8.center) -- (12345678.center) -- (1234567.center);
        \end{tikzpicture} \\
        \begin{tikzpicture}[scale=0.42]
            \node (8) at (0,0)[label=below:M]{};
            \node (7) at (2,0)[label=below:C]{};
            \node (4) at (4,0)[label=below:W]{};
            \node (1) at (6,0)[label=below:D]{};
            \node (2) at (10,0)[label=below:B]{};
            \node (3) at (8,0)[label=below:R]{};
            \node (5) at (12,0)[label=below:S]{};
            \node (6) at (14,0)[label=below:SL]{};
            \node (56) at (3+60/6,15/2-1/2*(3+60/6)[label=above:15]{};
            \node (256) at (3+50/6,15/2-1/2*(3+50/6)[label=above:33]{};
            \node (2356) at (3+40/6,15/2-1/2*(3+40/6)[label=above:35]{};
            \node (12356) at (3+30/6,15/2-1/2*(3+30/6)[label=above:41]{};
            \node (123456) at (3+20/6,15/2-1/2*(3+20/6)[label=above:42]{};
            \node (1234567) at (3 + 10/6,15/2-1/2*(3+10/6)[label=above:89]{};
            \node (12345678) at (3,6)[label=above:143]{};
            \draw (5.center) -- (56.center) -- (6.center);
            \draw (2.center) -- (256.center) -- (56.center);
            \draw (3.center) -- (2356.center) -- (256.center);
            \draw (1.center) -- (12356.center) -- (2356.center);
            \draw (4.center) -- (123456.center) -- (12356.center);
            \draw (7.center) -- (1234567.center) -- (123456.center);
            \draw (8.center) -- (12345678.center) -- (1234567.center);
        \end{tikzpicture}& &
        \begin{tikzpicture}[scale=0.42]
            \node (8) at (0,0)[label=below:M]{};
            \node (7) at (2,0)[label=below:C]{};
            \node (4) at (4,0)[label=below:W]{};
            \node (1) at (6,0)[label=below:D]{};
            \node (2) at (10,0)[label=below:B]{};
            \node (3) at (8,0)[label=below:R]{};
            \node (5) at (12,0)[label=below:S]{};
            \node (6) at (14,0)[label=below:SL]{};
            \node (56) at (3+60/6,15/2-1/2*(3+60/6)[label=above:15]{};
            \node (256) at (3+50/6,15/2-1/2*(3+50/6)[label=above:24]{};
            \node (2356) at (3+40/6,15/2-1/2*(3+40/6)[label=above:38]{};
            \node (12356) at (3+30/6,15/2-1/2*(3+30/6)[label=above:41]{};
            \node (123456) at (3+20/6,15/2-1/2*(3+20/6)[label=above:42]{};
            \node (1234567) at (3 + 10/6,15/2-1/2*(3+10/6)[label=above:89]{};
            \node (12345678) at (3,6)[label=above:143]{};
            \draw (5.center) -- (56.center) -- (6.center);
            \draw (2.center) -- (256.center) -- (56.center);
            \draw (3.center) -- (2356.center) -- (256.center);
            \draw (1.center) -- (12356.center) -- (2356.center);
            \draw (4.center) -- (123456.center) -- (12356.center);
            \draw (7.center) -- (1234567.center) -- (123456.center);
            \draw (8.center) -- (12345678.center) -- (1234567.center);
        \end{tikzpicture}\\
        \begin{tikzpicture}[scale=0.42]
            \node (8) at (0,0)[label=below:M]{};
            \node (7) at (2,0)[label=below:C]{};
            \node (4) at (4,0)[label=below:W]{};
            \node (1) at (6,0)[label=below:D]{};
            \node (2) at (10,0)[label=below:B]{};
            \node (3) at (8,0)[label=below:R]{};
            \node (5) at (12,0)[label=below:S]{};
            \node (6) at (14,0)[label=below:SL]{};
            \node (56) at (3+60/6,15/2-1/2*(3+60/6)[label=above:15]{};
            \node (256) at (3+50/6,15/2-1/2*(3+50/6)[label=above:24]{};
            \node (2356) at (3+40/6,15/2-1/2*(3+40/6)[label=above:35]{};
            \node (12356) at (3+30/6,15/2-1/2*(3+30/6)[label=above:41]{};
            \node (123456) at (3+20/6,15/2-1/2*(3+20/6)[label=above:43]{};
            \node (1234567) at (3 + 10/6,15/2-1/2*(3+10/6)[label=above:89]{};
            \node (12345678) at (3,6)[label=above:143]{};
            \draw (5.center) -- (56.center) -- (6.center);
            \draw (2.center) -- (256.center) -- (56.center);
            \draw (3.center) -- (2356.center) -- (256.center);
            \draw (1.center) -- (12356.center) -- (2356.center);
            \draw (4.center) -- (123456.center) -- (12356.center);
            \draw (7.center) -- (1234567.center) -- (123456.center);
            \draw (8.center) -- (12345678.center) -- (1234567.center);
        \end{tikzpicture}& &
        \begin{tikzpicture}[scale=0.42]
            \node (8) at (0,0)[label=below:M]{};
            \node (7) at (2,0)[label=below:C]{};
            \node (4) at (4,0)[label=below:W]{};
            \node (1) at (6,0)[label=below:D]{};
            \node (2) at (10,0)[label=below:B]{};
            \node (3) at (8,0)[label=below:R]{};
            \node (5) at (12,0)[label=below:S]{};
            \node (6) at (14,0)[label=below:SL]{};
            \node (56) at (3+60/6,15/2-1/2*(3+60/6)[label=above:15]{};
            \node (256) at (3+50/6,15/2-1/2*(3+50/6)[label=above:24]{};
            \node (2356) at (3+40/6,15/2-1/2*(3+40/6)[label=above:35]{};
            \node (12356) at (3+30/6,15/2-1/2*(3+30/6)[label=above:41]{};
            \node (123456) at (3+20/6,15/2-1/2*(3+20/6)[label=above:42]{};
            \node (1234567) at (3 + 10/6,15/2-1/2*(3+10/6)[label=above:93]{};
            \node (12345678) at (3,6)[label=above:143]{};
            \draw (5.center) -- (56.center) -- (6.center);
            \draw (2.center) -- (256.center) -- (56.center);
            \draw (3.center) -- (2356.center) -- (256.center);
            \draw (1.center) -- (12356.center) -- (2356.center);
            \draw (4.center) -- (123456.center) -- (12356.center);
            \draw (7.center) -- (1234567.center) -- (123456.center);
            \draw (8.center) -- (12345678.center) -- (1234567.center);
        \end{tikzpicture}\\
        \begin{tikzpicture}[scale=0.42]
            \node (8) at (0,0)[label=below:M]{};
            \node (7) at (2,0)[label=below:C]{};
            \node (4) at (4,0)[label=below:W]{};
            \node (1) at (6,0)[label=below:D]{};
            \node (2) at (10,0)[label=below:B]{};
            \node (3) at (8,0)[label=below:R]{};
            \node (5) at (12,0)[label=below:S]{};
            \node (6) at (14,0)[label=below:SL]{};
            \node (56) at (3+60/6,15/2-1/2*(3+60/6)[label=above:15]{};
            \node (256) at (3+50/6,15/2-1/2*(3+50/6)[label=above:24]{};
            \node (2356) at (3+40/6,15/2-1/2*(3+40/6)[label=above:35]{};
            \node (12356) at (3+30/6,15/2-1/2*(3+30/6)[label=above:41]{};
            \node (123456) at (3+20/6,15/2-1/2*(3+20/6)[label=above:42]{};
            \node (1234567) at (3 + 10/6,15/2-1/2*(3+10/6)[label=above:89]{};
            \node (12345678) at (3,6)[label=above:145]{};
            \draw (5.center) -- (56.center) -- (6.center);
            \draw (2.center) -- (256.center) -- (56.center);
            \draw (3.center) -- (2356.center) -- (256.center);
            \draw (1.center) -- (12356.center) -- (2356.center);
            \draw (4.center) -- (123456.center) -- (12356.center);
            \draw (7.center) -- (1234567.center) -- (123456.center);
            \draw (8.center) -- (12345678.center) -- (1234567.center);
        \end{tikzpicture}& &
        \begin{tikzpicture}[scale=0.42]
            \node (8) at (0,0)[label=below:M]{};
            \node (7) at (2,0)[label=below:C]{};
            \node (4) at (4,0)[label=below:W]{};
            \node (1) at (6,0)[label=below:D]{};
            \node (6) at (10,0)[label=below:SL]{};
            \node (3) at (8,0)[label=below:R]{};
            \node (2) at (12,0)[label=below:B]{};
            \node (5) at (14,0)[label=below:S]{};
            \node (25) at (3+60/6,15/2-1/2*(3+60/6)[label=above:20]{};
            \node (256) at (3+50/6,15/2-1/2*(3+50/6)[label=above:24]{};
            \node (2356) at (3+40/6,15/2-1/2*(3+40/6)[label=above:35]{};
            \node (12356) at (3+30/6,15/2-1/2*(3+30/6)[label=above:41]{};
            \node (123456) at (3+20/6,15/2-1/2*(3+20/6)[label=above:42]{};
            \node (1234567) at (3 + 10/6,15/2-1/2*(3+10/6)[label=above:89]{};
            \node (12345678) at (3,6)[label=above:143]{};
            \draw (5.center) -- (25.center) -- (2.center);
            \draw (6.center) -- (256.center) -- (25.center);
            \draw (3.center) -- (2356.center) -- (256.center);
            \draw (1.center) -- (12356.center) -- (2356.center);
            \draw (4.center) -- (123456.center) -- (12356.center);
            \draw (7.center) -- (1234567.center) -- (123456.center);
            \draw (8.center) -- (12345678.center) -- (1234567.center);
        \end{tikzpicture} \\
        \begin{tikzpicture}[scale=0.42]
            \node (8) at (0,0)[label=below:M]{};
            \node (7) at (2,0)[label=below:C]{};
            \node (4) at (4,0)[label=below:W]{};
            \node (1) at (6,0)[label=below:D]{};
            \node (6) at (10,0)[label=below:SL]{};
            \node (3) at (8,0)[label=below:R]{};
            \node (2) at (12,0)[label=below:B]{};
            \node (5) at (14,0)[label=below:S]{};
            \node (25) at (3+60/6,15/2-1/2*(3+60/6)[label=above:20]{};
            \node (256) at (3+50/6,15/2-1/2*(3+50/6)[label=above:33]{};
            \node (2356) at (3+40/6,15/2-1/2*(3+40/6)[label=above:35]{};
            \node (12356) at (3+30/6,15/2-1/2*(3+30/6)[label=above:41]{};
            \node (123456) at (3+20/6,15/2-1/2*(3+20/6)[label=above:42]{};
            \node (1234567) at (3 + 10/6,15/2-1/2*(3+10/6)[label=above:89]{};
            \node (12345678) at (3,6)[label=above:143]{};
            \draw (5.center) -- (25.center) -- (2.center);
            \draw (6.center) -- (256.center) -- (25.center);
            \draw (3.center) -- (2356.center) -- (256.center);
            \draw (1.center) -- (12356.center) -- (2356.center);
            \draw (4.center) -- (123456.center) -- (12356.center);
            \draw (7.center) -- (1234567.center) -- (123456.center);
            \draw (8.center) -- (12345678.center) -- (1234567.center);
        \end{tikzpicture}& &
        \begin{tikzpicture}[scale=0.42]
            \node (8) at (0,0)[label=below:M]{};
            \node (7) at (2,0)[label=below:C]{};
            \node (4) at (4,0)[label=below:W]{};
            \node (1) at (6,0)[label=below:D]{};
            \node (6) at (10,0)[label=below:SL]{};
            \node (3) at (8,0)[label=below:R]{};
            \node (2) at (12,0)[label=below:B]{};
            \node (5) at (14,0)[label=below:S]{};
            \node (25) at (3+60/6,15/2-1/2*(3+60/6)[label=above:20]{};
            \node (256) at (3+50/6,15/2-1/2*(3+50/6)[label=above:24]{};
            \node (2356) at (3+40/6,15/2-1/2*(3+40/6)[label=above:38]{};
            \node (12356) at (3+30/6,15/2-1/2*(3+30/6)[label=above:41]{};
            \node (123456) at (3+20/6,15/2-1/2*(3+20/6)[label=above:42]{};
            \node (1234567) at (3 + 10/6,15/2-1/2*(3+10/6)[label=above:89]{};
            \node (12345678) at (3,6)[label=above:143]{};
            \draw (5.center) -- (25.center) -- (2.center);
            \draw (6.center) -- (256.center) -- (25.center);
            \draw (3.center) -- (2356.center) -- (256.center);
            \draw (1.center) -- (12356.center) -- (2356.center);
            \draw (4.center) -- (123456.center) -- (12356.center);
            \draw (7.center) -- (1234567.center) -- (123456.center);
            \draw (8.center) -- (12345678.center) -- (1234567.center);
        \end{tikzpicture}\\
        \begin{tikzpicture}[scale=0.42]
            \node (8) at (0,0)[label=below:M]{};
            \node (7) at (2,0)[label=below:C]{};
            \node (4) at (4,0)[label=below:W]{};
            \node (1) at (6,0)[label=below:D]{};
            \node (6) at (10,0)[label=below:SL]{};
            \node (3) at (8,0)[label=below:R]{};
            \node (2) at (12,0)[label=below:B]{};
            \node (5) at (14,0)[label=below:S]{};
            \node (25) at (3+60/6,15/2-1/2*(3+60/6)[label=above:20]{};
            \node (256) at (3+50/6,15/2-1/2*(3+50/6)[label=above:24]{};
            \node (2356) at (3+40/6,15/2-1/2*(3+40/6)[label=above:35]{};
            \node (12356) at (3+30/6,15/2-1/2*(3+30/6)[label=above:41]{};
            \node (123456) at (3+20/6,15/2-1/2*(3+20/6)[label=above:43]{};
            \node (1234567) at (3 + 10/6,15/2-1/2*(3+10/6)[label=above:89]{};
            \node (12345678) at (3,6)[label=above:143]{};
            \draw (5.center) -- (25.center) -- (2.center);
            \draw (6.center) -- (256.center) -- (25.center);
            \draw (3.center) -- (2356.center) -- (256.center);
            \draw (1.center) -- (12356.center) -- (2356.center);
            \draw (4.center) -- (123456.center) -- (12356.center);
            \draw (7.center) -- (1234567.center) -- (123456.center);
            \draw (8.center) -- (12345678.center) -- (1234567.center);
        \end{tikzpicture}& &
        \begin{tikzpicture}[scale=0.42]
            \node (8) at (0,0)[label=below:M]{};
            \node (7) at (2,0)[label=below:C]{};
            \node (4) at (4,0)[label=below:W]{};
            \node (1) at (6,0)[label=below:D]{};
            \node (6) at (10,0)[label=below:SL]{};
            \node (3) at (8,0)[label=below:R]{};
            \node (2) at (12,0)[label=below:B]{};
            \node (5) at (14,0)[label=below:S]{};
            \node (25) at (3+60/6,15/2-1/2*(3+60/6)[label=above:20]{};
            \node (256) at (3+50/6,15/2-1/2*(3+50/6)[label=above:24]{};
            \node (2356) at (3+40/6,15/2-1/2*(3+40/6)[label=above:35]{};
            \node (12356) at (3+30/6,15/2-1/2*(3+30/6)[label=above:41]{};
            \node (123456) at (3+20/6,15/2-1/2*(3+20/6)[label=above:42]{};
            \node (1234567) at (3 + 10/6,15/2-1/2*(3+10/6)[label=above:93]{};
            \node (12345678) at (3,6)[label=above:143]{};
            \draw (5.center) -- (25.center) -- (2.center);
            \draw (6.center) -- (256.center) -- (25.center);
            \draw (3.center) -- (2356.center) -- (256.center);
            \draw (1.center) -- (12356.center) -- (2356.center);
            \draw (4.center) -- (123456.center) -- (12356.center);
            \draw (7.center) -- (1234567.center) -- (123456.center);
            \draw (8.center) -- (12345678.center) -- (1234567.center);
        \end{tikzpicture}\\
        \begin{tikzpicture}[scale=0.42]
            \node (8) at (0,0)[label=below:M]{};
            \node (7) at (2,0)[label=below:C]{};
            \node (4) at (4,0)[label=below:W]{};
            \node (1) at (6,0)[label=below:D]{};
            \node (6) at (10,0)[label=below:SL]{};
            \node (3) at (8,0)[label=below:R]{};
            \node (2) at (12,0)[label=below:B]{};
            \node (5) at (14,0)[label=below:S]{};
            \node (25) at (3+60/6,15/2-1/2*(3+60/6)[label=above:20]{};
            \node (256) at (3+50/6,15/2-1/2*(3+50/6)[label=above:24]{};
            \node (2356) at (3+40/6,15/2-1/2*(3+40/6)[label=above:35]{};
            \node (12356) at (3+30/6,15/2-1/2*(3+30/6)[label=above:41]{};
            \node (123456) at (3+20/6,15/2-1/2*(3+20/6)[label=above:42]{};
            \node (1234567) at (3 + 10/6,15/2-1/2*(3+10/6)[label=above:89]{};
            \node (12345678) at (3,6)[label=above:145]{};
            \draw (5.center) -- (25.center) -- (2.center);
            \draw (6.center) -- (256.center) -- (25.center);
            \draw (3.center) -- (2356.center) -- (256.center);
            \draw (1.center) -- (12356.center) -- (2356.center);
            \draw (4.center) -- (123456.center) -- (12356.center);
            \draw (7.center) -- (1234567.center) -- (123456.center);
            \draw (8.center) -- (12345678.center) -- (1234567.center);
        \end{tikzpicture}& &
        \begin{tikzpicture}[scale=0.42]
            \node (8) at (0,0)[label=below:M]{};
            \node (7) at (2,0)[label=below:C]{};
            \node (4) at (4,0)[label=below:W]{};
            \node (1) at (6,0)[label=below:D]{};
            \node (5) at (10,0)[label=below:S]{};
            \node (3) at (8,0)[label=below:R]{};
            \node (2) at (12,0)[label=below:B]{};
            \node (6) at (14,0)[label=below:SL]{};
            \node (26) at (3+60/6,15/2-1/2*(3+60/6)[label=above:24]{};
            \node (256) at (3+50/6,15/2-1/2*(3+50/6)[label=above:33]{};
            \node (2356) at (3+40/6,15/2-1/2*(3+40/6)[label=above:35]{};
            \node (12356) at (3+30/6,15/2-1/2*(3+30/6)[label=above:41]{};
            \node (123456) at (3+20/6,15/2-1/2*(3+20/6)[label=above:42]{};
            \node (1234567) at (3 + 10/6,15/2-1/2*(3+10/6)[label=above:89]{};
            \node (12345678) at (3,6)[label=above:143]{};
            \draw (6.center) -- (26.center) -- (2.center);
            \draw (5.center) -- (256.center) -- (26.center);
            \draw (3.center) -- (2356.center) -- (256.center);
            \draw (1.center) -- (12356.center) -- (2356.center);
            \draw (4.center) -- (123456.center) -- (12356.center);
            \draw (7.center) -- (1234567.center) -- (123456.center);
            \draw (8.center) -- (12345678.center) -- (1234567.center);
        \end{tikzpicture}
    \end{longtable}}
    \end{center}

\bibliography{trees}

\begin{thebibliography}{10}

\bibitem{akian2011best}
Marianne Akian, St{\'e}phane Gaubert, Viorel Ni{\c{t}}ic{\u{a}}, and Ivan
  Singer.
\newblock Best approximation in max-plus semimodules.
\newblock {\em Linear Algebra and its Applications}, 435(12):3261--3296, 2011.

\bibitem{allamigeon-gaubert-goubault2010}
Xavier Allamigeon, St{\'e}phane Gaubert, and {\'E}ric Goubault.
\newblock The tropical double description method.
\newblock In {\em S{TACS} 2010: 27th {I}nternational {S}ymposium on
  {T}heoretical {A}spects of {C}omputer {S}cience}, volume~5 of {\em LIPIcs.
  Leibniz Int. Proc. Inform.}, pages 47--58. Schloss Dagstuhl. Leibniz-Zent.
  Inform., Wadern, 2010.

\bibitem{ardila2004}
Federico Ardila.
\newblock Subdominant matroid ultrametrics.
\newblock {\em Annals of Combinatorics}, 8:379--389, 2004.

\bibitem{ardila-klivans2006}
Federico Ardila and Caroline~J. Klivans.
\newblock The {B}ergman complex of a matroid and phylogenetic trees.
\newblock {\em Journal of Combinatorial Theory, Series B}, 96(1):38 -- 49,
  2006.

\bibitem{bernstein2017infinity}
Daniel~Irving Bernstein and Colby Long.
\newblock L-infinity optimization to linear spaces and phylogenetic trees.
\newblock {\em SIAM Journal on Discrete Mathematics}, 31(2):875--889, 2017.

\bibitem{billera2001geometry}
Louis~J Billera, Susan~P Holmes, and Karen Vogtmann.
\newblock Geometry of the space of phylogenetic trees.
\newblock {\em Advances in Applied Mathematics}, 27(4):733--767, 2001.

\bibitem{butkovivc2007generators}
Peter Butkovi{\v{c}}, Hans Schneider, and Sergei Sergeevc.
\newblock Generators, extremals and bases of max cones.
\newblock {\em Linear algebra and its applications}, 421(2-3):394--406, 2007.

\bibitem{chepoi}
V.~Chepoi and B.~Fichet.
\newblock $l_\infty$-approximation via subdominants.
\newblock {\em Journal of Mathematical Psychology}, 44:600--616, 2000.

\bibitem{develin-sturmfels2004}
Mike Develin and Bernd Sturmfels.
\newblock Tropical convexity.
\newblock {\em Documenta Mathematica}, 9:1--27, 2004.

\bibitem{farach1993robust}
Martin Farach, Sampath Kannan, and Tandy Warnow.
\newblock A robust model for finding optimal evolutionary trees.
\newblock In {\em Proceedings of the twenty-fifth annual ACM symposium on
  Theory of computing}, pages 137--145. ACM, 1993.

\bibitem{feichtner-sturmfels2005}
Eva~Maria Feichtner and Bernd Sturmfels.
\newblock Matroid polytopes, nested sets and bergman fans.
\newblock {\em Portugaliae Mathematica}, 62(4):437--468, 2005.

\bibitem{felsenstein2004inferring}
Joseph Felsenstein.
\newblock {\em Inferring phylogenies}, volume~2.
\newblock Sinauer associates Sunderland, 2004.

\bibitem{gaubert1992theorie}
St{\'e}phane Gaubert.
\newblock {\em Th{\'e}orie des syst{\`e}mes lin{\'e}aires dans les
  dio{\"\i}des}.
\newblock PhD thesis, Paris, ENMP, 1992.

\bibitem{gaubert-katz2007}
St{\'e}phane Gaubert and Ricardo~D Katz.
\newblock The {Minkowski} theorem for max-plus convex sets.
\newblock {\em Linear Algebra and its Applications}, 421(2):356--369, 2007.

\bibitem{gaubert2011minimal}
St{\'e}phane Gaubert and Ricardo~D Katz.
\newblock Minimal half-spaces and external representation of tropical
  polyhedra.
\newblock {\em Journal of Algebraic Combinatorics}, 33(3):325--348, 2011.

\bibitem{lin2018tropical}
Bo~Lin, Anthea Monod, and Ruriko Yoshida.
\newblock Tropical foundations for probability \& statistics on phylogenetic
  tree space.
\newblock {\em arXiv preprint arXiv:1805.12400}, 2018.

\bibitem{lin-sturmfels2016}
Bo~Lin, Bernd Sturmfels, Xiaoxian Tang, and Ruriko Yoshida.
\newblock Convexity in tree spaces.
\newblock {\em SIAM Journal on Discrete Mathematics}, 31(3):2015--2038, 2017.

\bibitem{lin2018tropicalFermat}
Bo~Lin and Ruriko Yoshida.
\newblock Tropical fermat--weber points.
\newblock {\em SIAM Journal on Discrete Mathematics}, 32(2):1229--1245, 2018.

\bibitem{oxley2011}
James Oxley.
\newblock {\em Matroid Theory}.
\newblock Oxford University Press, second edition, 2011.

\bibitem{sarich1969pinniped}
Vincent~M Sarich.
\newblock Pinniped phylogeny.
\newblock {\em Systematic Biology}, 18(4):416--422, 1969.

\bibitem{Semple2003}
Charles Semple and Mike Steel.
\newblock {\em Phylogenetics}.
\newblock Oxford University Press, Oxford, 2003.

\bibitem{Speyer}
David Speyer and Bernd Sturmfels.
\newblock The tropical {G}rassmannian.
\newblock {\em Adv. Geom.}, 4:389--411, 2004.

\bibitem{yoshida-zhang2017tropical}
Ruriko Yoshida, Leon Zhang, and Xu~Zhang.
\newblock Tropical principal component analysis and its application to
  phylogenetics.
\newblock {\em Bulletin of mathematical biology}, 81(2):568--597, 2019.

\bibitem{yu2019extreme}
Luyan Yu.
\newblock Extreme rays of the $l^\infty$-nearest ultrametric tropical polytope.
\newblock {\em arXiv preprint arXiv:1907.10521}, 2019.

\end{thebibliography}
\bibliographystyle{plain}

\end{document}